\documentclass[11pt]{article}
\usepackage[utf8]{inputenc}
\usepackage[english]{babel}
\usepackage{authblk}
\usepackage{graphicx}        
\usepackage{subfigure}                            
\usepackage{multicol}        
\usepackage[bottom]{footmisc}
\usepackage{placeins}
\usepackage{changes,todonotes}
\usepackage{pgf}
\usepackage{psfrag}
\usepackage{pgfplots}
\usepackage{hyperref}
\usepackage{appendix}
\usepackage{amsmath,amssymb}
\usepackage{bm}
\usepackage{mathtools}
\usepackage{color}
\usepackage{stmaryrd}
\usepackage{amsthm}
\usepackage{blindtext}
\usepackage{caption}
\usepackage{multirow}
\usepackage{float}

\usepackage[margin=3cm]{geometry}

\usepackage[ruled,vlined]{algorithm2e}

\graphicspath{./FIGURES/}
\newcommand{\xx}{\bm{x}}
\newcommand{\vv}{\bm{v}}

\newcommand{\uu}{\bm{u}}
\newcommand{\ww}{\bm{w}}
\newcommand{\zz}{\bm{z}}
\newcommand{\FF}{\bm{F}}
\newcommand{\ee}{\bm{e}}
\renewcommand{\aa}{\bm{a}}
\newcommand{\bb}{\bm{b}}
\newcommand{\rr}{\bm{r}}
\newcommand{\ff}{\bm{f}}

\newcommand{\UU}{\bm{U}}
\newcommand{\WW}{\bm{W}}
\newcommand{\ZZ}{\bm{Z}}
\newcommand{\GG}{\bm{G}}
\newcommand{\LL}{\bm{L}}

\newcommand{\HH}{\bm{H}}

\newcommand{\bsig}{\boldsymbol{\sigma}}

\newcommand{\vertiii}[1]{{\left\vert\kern-0.25ex\left\vert\kern-0.25ex\left\vert #1 
    \right\vert\kern-0.25ex\right\vert\kern-0.25ex\right\vert}}

\definecolor{farbe}{gray}{0.80}

\newtheorem{myth}{Theorem}
\newtheorem{myprop}{Proposition}
\newtheorem{mydef}{Definition}
\newtheorem{mylemma}{Lemma}

\newcommand{\bu}{\bold{u}}

\newcommand{\bw}{\bold{w}}

\begin{document}

\title{A discontinuous Galerkin time integration scheme for second order differential equations with applications to seismic wave propagation problems}

\author[$\star$]{Paola F. Antonietti}
\author[$\star$]{Ilario Mazzieri}
\author[$\star$]{Francesco Migliorini}

\affil[$\star$]{MOX, Laboratory for Modeling and Scientific Computing, Dipartimento di Matematica, Politecnico di Milano, Piazza Leonardo da Vinci 32, I-20133 Milano, Italy}

\affil[ ]{\texttt {\{paola.antonietti,ilario.mazzieri,francesco.migliorini\}@polimi.it}}



\maketitle

	
\begin{abstract}
In this work, we present a new high order Discontinuous Galerkin time integration scheme for second-order (in time) differential systems that typically arise from the space discretization of the elastodynamics equation. 
By rewriting the original equation as a system of first order differential equations we introduce the method and show that the resulting discrete  formulation is well-posed, stable and retains super-optimal rate of convergence with respect to the discretization parameters, namely the time step and the polynomial approximation degree. A set of two- and three-dimensional numerical experiments confirm the theoretical bounds. Finally, the method is applied to real geophysical applications.

\end{abstract}

\section{Introduction}

In this paper we present and analyze a high-order time discontinuous Galerkin finite element method for the time integration of second order differential problems as those stemming from e.g. elastic wave propagation phenomena. 

Classical approaches for the time integration of second order differential systems employ implicit and explicit finite differences, Leap-frog, Runge-Kutta or Newmark schemes, see e.g. \cite{Ve07,Bu08,QuSaSa07} for a detailed review. In computational seismology, explicit time integration schemes are nowadays preferred to implicit ones, due to their computational cheapness and ease of implementation. Indeed, although being unconditionally stable, implicit methods are typically computationally expensive. The main drawback of explicit methods is that they are conditionally stable and the choice of time step imposed by the Courant-Freidrichs-Levy (CFL) condition can sometimes be a great limitation. 

To overcome this limitation one can employ local time stepping (LTS) algorithms \cite{GrMi13,DiGr09,CoFoJo03,Dumbser2007arbitrary} for which the CFL condition is imposed element-wise leading to an optimal choice of the time step. The unique drawback of this approach is the additional synchronization process that one need to take into account for a correct propagation of the wave field from one element to the other.

In this work, we present an implicit time integration method based on a discontinuous Galerkin (DG) approach. Originally, DG methods \cite{ReedHill73,Lesaint74} have been developed to approximate \textit{in space} hyperbolic problems \cite{ReedHill73}, and then generalized to elliptic and parabolic equations \cite{wheeler1978elliptic,arnold1982interior,HoScSu00,CockKarnShu00,
riviere2008discontinuous,HestWar,DiPiEr}. We refer the reader to \cite{riviere2003discontinuous,Grote06} for the application of DG methods to scalar wave equations and to \cite{Dumbser2007arbitrary,WiSt2010,antonietti2012non,
ferroni2016dispersion,antonietti2016stability,AnMa2018,
Antonietti_etal2018,mazzieri2013speed,AnMaMi20,DeGl15} for the elastodynamics problem. 

The DG approach has been used also to approximate initial-value problem where the DG paradigm shows some advantage with respect to other implicit schemes such as the Johnson's method, see e.g. \cite{JOHNSON1993,ADJERID2011}. Indeed, since the information follows the positive direction of time, the solution  at time-slab $[t_n,t_{n+1}]$ depends only on the solution at the time instant $t_n^-$. 
By employing DG methods in both space and time dimensions it leads to a fully DG space-time formulation such as \cite{Delfour81,Vegt2006,WeGeSc2001,AnMaMi20}.

More generally, space-time methods have been largely employed for hyperbolic problems. Indeed, high order approximations in both space and time are simple to obtain, achieving spectral convergence of the space-time error through $p$-refinement. In addition, stability can be achieved with local CFL conditions, as in \cite{MoRi05}, increasing computational efficiency. 
Space-time methods can be divided according to which type of space-time partition they employ. In structured techniques \cite{CangianiGeorgoulisHouston_2014,Tezduyar06}, the space-time grid is the cartesian product of a spatial mesh and a time partition. Examples of applications to second order hyperbolic problems can be found in \cite{StZa17,ErWi19,BaMoPeSc20}. Unstructured techniques \cite{Hughes88,Idesman07} employ grids generated considering the time as an additional dimension. See \cite{Yin00,AbPeHa06,DoFiWi16} for examples of applications to first order hyperbolic problems. Unstructured methods may have better properties, however they suffer from the difficulty of generating the mesh, especially for three-dimensional problems.
Among unstructured methods, we mention Trefftz techniques \cite{KrMo16,BaGeLi17,BaCaDiSh18}, in which the numerical solution is looked for in the Trefftz space, and the tent-pitching paradigm \cite{GoScWi17}, in which the space-time elements are progressively built on top of each other in order to grant stability of the numerical scheme. Recently, in \cite{MoPe18,PeScStWi20} a combination of Trefftz and tent-pitching techniques has been proposed with application to first order hyperbolic problems.
Finally, a typical approach for second order differential equations consists in reformulating them as a system of first order hyperbolic equations. Thus, velocity is considered as an additional problem's unkwnown that results in doubling the dimension of the final linear system, cf. \cite{Delfour81,Hughes88,FRENCH1993,JOHNSON1993,ThHe2005}.

 The motivation for this work is to overcome the limitations of the space-time DG method presented in \cite{AnMaMi20} for elastodynamics problems. This method integrates the second order (in time) differential problem stemming from the spatial discretization. The resulting stiffness matrix is ill-conditioned making the use of iterative solvers quite difficult. Hence, direct methods are used forcing to store the stiffness matrix and greatly reducing the range of problems affordable by that method. Here, we propose to change the way the time integration is obtained, resulting in a well-conditioned system matrix and making iterative methods employable and complex 3D problems solvable. 

In this work, we present a high order discontinuous Galerkin method for time integration of systems of second-order differential equations stemming from space discretization of the visco-elastodynamics problem. The differential (in time) problem is firstly reformulated as a first order system, then, by imposing only weak continuity of tractions across time slabs, we derive a discontinuous Galerkin method. We show the well posedness of the proposed method through the definition of a suitable energy norm, and we prove  stability and \emph{a priori} error estimates. The obtained scheme is implicit, unconditionally stable and super-optimal in term of accuracy with respect to the integration time step. In addition, the solution strategy adopted for the associated algebraic linear system reduces the complexity and computational cost of the solution, making three dimensional problems (in space) affordable.

The paper is organized as follows. In Section \ref{Sc:Method} we formulate the problem, present its numerical discretization and show that it is well-posed. The stability and convergence properties of the method are discussed in Section \ref{Sc:Convergence}, where we present \textit{a priori} estimates in a suitable norm. In Section \ref{Sc:AlgebraicFormulation}, the equations are rewritten into the corresponding algebraic linear system and a suitable solution strategy is shown. Finally, in Section \ref{Sc:NumericalResults}, the method is validated through several numerical experiments both in two and three dimensions.

Throughout the paper, we denote by $||\aa||$ the Euclidean norm of a vector $\aa \in \mathbb{R}^d$, $d\ge 1$ and by $||A||_{\infty} = \max_{i=1,\dots,m}\sum_{j=1}^n |a_{ij}|$, the $\ell^{\infty}$-norm of a matrix $A\in\mathbb{R}^{m\times n}$, $m,n\ge1$. For a given $I\subset\mathbb{R}$ and $v:I\rightarrow\mathbb{R}$ we denote by $L^p(I)$ and $H^p(I)$, $p\in\mathbb{N}_0$, the classical Lebesgue and Hilbert spaces, respectively, and endow them with the usual norms, see \cite{AdamsFournier2003}. Finally, we indicate the Lebesgue and Hilbert spaces for vector-valued functions as $\LL^p(I) = [L^p(I)]^d$ and $\HH^p(I) = [H^p(I)]^d$, $d\ge1$, respectively.

\section{Discontinuous Galerkin approximation of a second-order initial value problem}
\label{Sc:Method}
For $T>0$, we consider the following model problem \cite{kroopnick}: find $\uu(t) \in\HH^2(0,T]$ such that 
\begin{equation}
	\label{Eq:SecondOrderEquation}
	\begin{cases}
		P\ddot{\uu}(t) + L\dot{\uu}(t)+K\uu(t) = \ff(t) \qquad \forall\, t \in (0,T], \\
		\uu(0) = \hat{\uu}_0, \\
		\dot{\uu}(0) = \hat{\uu}_1,
	\end{cases}
\end{equation}
where $P,L,K \in \mathbb{R}^{d\times d}$, $d\geq 1$ are symmetric, positive definite matrices, $\hat{\uu}_0, \hat{\uu}_1 \in \mathbb{R}^d$ and $\ff \in \LL^2(0,T]$. Then, we introduce a variable $\ww:(0,T]\rightarrow\mathbb{R}^{d}$ that is the first derivative of $\uu$, i.e. $\ww(t) = \dot{\uu}(t)$, and reformulate problem \eqref{Eq:SecondOrderEquation} as a system of first order differential equations:
\begin{equation}
\label{Eq:FirstOrderSystem1}
	\begin{cases}
		K\dot{\uu}(t) - K\ww(t) = \boldsymbol{0} &\forall\, t\in(0,T], \\
		P\dot{\ww}(t) +L\ww(t) + K\uu(t) = \ff(t) &\forall\, t\in(0,T], \\
		\uu(0) = \hat{\uu}_0, \\
		\ww(0) = \hat{\uu}_1.
	\end{cases}
\end{equation}
Note that, since $K$ is a positive definite matrix, the first equation in \eqref{Eq:FirstOrderSystem1} is consistent with the definition of $\ww$. By defining $\zz = [\uu,\ww]^T\in\mathbb{R}^{2d}$, $\FF=[\bm{0},\ff]^T\in\mathbb{R}^{2d}$, $\zz_0 = [\hat{\uu}_0,\hat{\uu}_1]^T\in\mathbb{R}^{2d}$ and
\begin{equation}\label{def:KA}
	\widetilde{K} = \begin{bmatrix}
		K & 0 \\
		0 & P
	\end{bmatrix}\in\mathbb{R}^{2d\times2d}, \quad
 	A = \begin{bmatrix}
	 	0 & -K \\
		K & L
	\end{bmatrix}\in\mathbb{R}^{2d\times2d},
\end{equation}
we can write \eqref{Eq:FirstOrderSystem1} as
\begin{equation}
	\label{Eq:FirstOrderSystem2}
	\begin{cases}
		\tilde{K}\dot{\zz}(t) + A\zz(t) = \FF(t) & \forall\, t\in(0,T], \\
		\zz(0) = \zz_0.
	\end{cases}
\end{equation}

To integrate in time system \eqref{Eq:FirstOrderSystem2}, we first partition the interval $I=(0,T]$ into $N$ time-slabs $I_n = (t_{n-1},t_n]$ having length $\Delta t_n = t_n-t_{n-1}$, for $n=1,\dots,N$ with $t_0 = 0$ and $t_N = T$, as it is shown in Figure \ref{Fig:TimeDomain}.
\begin{figure}[h!]
	\centering
\includegraphics[width=1\textwidth]{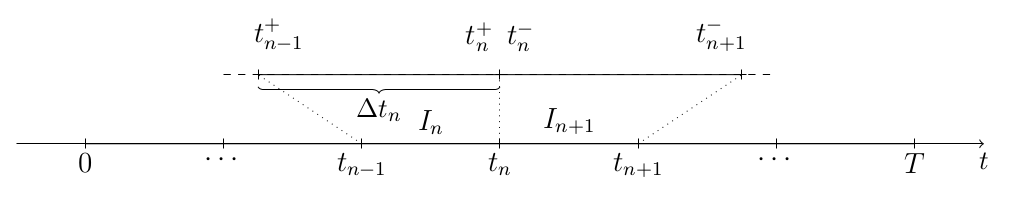}	
	\caption{Example of time domain partition (bottom). Zoom of the time domain partition: values $t_n^+$ and $t_n^-$ are also reported (top).}\label{Fig:TimeDomain}
\end{figure}

Next, we incrementally build (on $n$) an approximation of the exact solution $\uu$ in each time slab $I_n$. In the following we will use the notation
\begin{equation*}
	(\uu,\vv)_I = \int_I \uu(s)\cdot\vv(s)\text{d}s, \quad \langle \uu,\vv \rangle_t = \uu(t)\cdot \vv(t),
\end{equation*}
where $\aa\cdot\bb$ stands for the euclidean scalar product between tho vectors $\aa,\bb\in\mathbb{R}^d$. We also denote  for (a regular enough) $\vv$, the jump operator at $t_n$ as
\begin{equation*}
	[\vv]_n = \vv(t_n^+) - \vv(t_n^-) = \vv^+ -\vv^-, \quad \text{for } n\ge 0,
\end{equation*}
where
\begin{equation*}
	\vv(t_n^\pm) = \lim_{\epsilon\rightarrow 0^\pm}\vv(t_n+\epsilon), \quad \text{for } n\ge 0.
\end{equation*}

 Thus, we focus on the generic interval $I_n$ and assume that the solution on $I_{n-1}$ is known. We multiply equation \eqref{Eq:FirstOrderSystem2} by a (regular enough) test function $\vv(t)\in\mathbb{R}^{2d}$ and integrate in time over $I_n$ obtaining
\begin{equation}
	\label{Eq:Weak1}
	(\widetilde{K}\dot{\zz},\vv)_{I_n} + (A\zz,\vv)_{I_n}  = (\FF,\vv)_{I_n}.
\end{equation}
Next, since $\uu \in\HH^2(0,T]$ and $\ww = \dot{\uu}$, then $\zz\in\HH^1(0,T]$. Therefore, we can add to \eqref{Eq:Weak1} the null term $\widetilde{K}[\zz]_{n-1}\cdot\vv(t_{n-1}^+)$ getting
\begin{equation}
	\label{Eq:Weak2}
	(\widetilde{K}\dot{\zz},\vv)_{I_n} + (A\zz,\vv)_{I_n} +\widetilde{K}[\zz]_{n-1}\cdot\vv(t_{n-1}^+) = (\FF,\vv)_{I_n}.
\end{equation}
Summing up over all time slabs we define the bilinear form $\mathcal{A}:\HH^1(0,T)\times\HH^1(0,T)\rightarrow\mathbb{R}$
\begin{equation}
	\label{Eq:BilinearForm}
	\mathcal{A}(\zz,\vv) = \sum_{n=1}^N (\widetilde{K}\dot{\zz},\vv)_{I_n} + (A\zz,\vv)_{I_n} + \sum_{n=1}^{N-1} \widetilde{K}[\zz]_n\cdot\vv(t_n^+) +  \widetilde{K}\zz(0^+)\cdot\vv(0^+),
\end{equation}
and the linear functional $\mathcal{F}:\LL^2(0,T)\rightarrow\mathbb{R}$ as
\begin{equation}
	\label{Eq:LinearFunctional}
	\mathcal{F}(\vv) = \sum_{n=1}^N (\FF,\vv)_{I_n} + \widetilde{K}\zz_0\cdot\vv_{0}^+,
\end{equation}
where we have used that $\zz(0^-) = \zz_0$. Now, we introduce the functional spaces
\begin{equation}
	\label{Eq:PolynomialSpace}
	V_n^{r_n} = \{ \zz:I_n\rightarrow\mathbb{R}^{2d} \text{ s.t. }  \zz\in[\mathcal{P}^{r_n}(I_n)]^{2d} \},
\end{equation}
where $\mathcal{P}^{r_n}(I_n)$ is the space of polynomial defined on $I_n$ of maximum degree $r_n$, 
\begin{equation}
	\label{Eq:L2Space}
	\mathcal{V}^{\rr} = \{ \zz\in\LL^2(0,T] \text{ s.t. } \zz|_{I_n} = [\uu,\ww]^T\in V_n^{r_n} \},
\end{equation}
and 
\begin{equation}
	\label{Eq:CGSpace}
	\mathcal{V}_{CG}^{\rr} = \{ \zz\in[\mathbb{C}^0(0,T]]^{2d} \text{ s.t. } \zz|_{I_n} = [\uu,\ww]^T\in V_n^{r_n} \text{ and } \dot{\uu} = \ww \},
\end{equation}
where $\rr = (r_1,\dots,r_N) \in \mathbb{N}^N$ is the polynomial degree vector

Before assessing the discontinuous Galerkin formulation of problem~\eqref{Eq:FirstOrderSystem2}, we need to introduce, as in \cite{ScWi2010}, the following operator $\mathcal{R}$, that is used only on the purpose of the analysis and does not need to be computed in practice.
\begin{mydef}
	\label{Def:Reconstruction}
	We define a reconstruction operator $\mathcal{R}:\mathcal{V}^{\rr}\rightarrow\mathcal{V}^{\rr}_{CG}$ such that
	\begin{equation}
	\label{Eq:Reconstruction}
	\begin{split}
	(\mathcal{R}'(\zz),\vv)_{I_n} &= (\zz',\vv)_{I_n} + [\zz]_{n-1}\cdot\vv(t_{n-1}^+) \quad \forall\, \vv\in[\mathcal{P}^{r_n}(I_n)]^{2d}, \\ \mathcal{R}(\zz(t_{n-1}^+)) &= \zz(t_{n-1}^-) \quad \forall\, n =1,\dots,N.
	\end{split}
	\end{equation}
\end{mydef}

\noindent Now, we can properly define the  functional space
\begin{equation}
	\label{Eq:DGSpace}
	\begin{split}
	\mathcal{V}_{DG}^{\rr} = \{& \zz\in\mathcal{V}^{\rr} \text{ and }\exists\, \hat{\zz} = R(\zz) \in\mathcal{V}_{CG}^{\rr}\},
	\end{split}
\end{equation} 
and introduce the DG formulation of \eqref{Eq:FirstOrderSystem2} reads as follows. Find $\zz_{DG}\in\mathcal{V}_{DG}^{\rr}$ such that
\begin{equation}
	\label{Eq:WeakProblem}
	\mathcal{A}(\zz_{DG},\vv) = \mathcal{F}(\vv) \qquad \vv\in\mathcal{V}_{DG}^{\rr}.
\end{equation}

For the forthcoming analysis we introduce the following mesh-dependent energy norm.
\begin{myprop}
	\label{Pr:Norm}
	The function $|||\cdot|||:\mathcal{V}_{DG}^{\rr}\rightarrow\mathbb{R}^{+}$, is defined as 
	\begin{equation}
		\label{Eq:Norm}
		|||\zz|||^2 = \sum_{n=1}^N ||\widetilde{L}\zz||_{\LL^2(I_n)}^2 + \frac{1}{2}(\widetilde{K}^{\frac{1}{2}}\zz(0^+))^2 + \frac{1}{2}\sum_{n=1}^{N-1}(\widetilde{K}^{\frac{1}{2}}[\zz]_n)^2 + \frac{1}{2}(\widetilde{K}^{\frac{1}{2}}\zz(T^-))^2,
	\end{equation}
	with
	$
	\widetilde{L} = \begin{bmatrix}
	0 & 0 \\
	0 & L^{\frac{1}{2}}
	\end{bmatrix}\in\mathbb{R}^{2d\times2d}.
$
	Moreover a norm on $\mathcal{V}_{DG}^{\rr}$.
\end{myprop}
\begin{proof}
	It is clear that homogeneity and subadditivity hold. In addition, it is trivial that if $\zz = 0$ then $|||\zz|||=0$. Therefore, we suppose $|||\zz||| = 0$ and observe that 
	\begin{equation*}
		||\widetilde{L}\zz||_{\LL^2(I_n)}=||L^{\frac{1}{2}}\ww||_{\LL^2(I_n)}=0 \quad \forall n=1,\dots,N.
	\end{equation*}
	Since $L$ is positive definite we have $\ww = \textbf{0} $ on $[0,T]$. Hence, $\ww'=\textbf{0}$ on $[0,T]$. Using this result into \eqref{Eq:DGSpace} and calling $\vv = [\vv_1,\vv_2]^T$, we get 
	\begin{equation*}
		(\hat{\ww}',\vv_2)_{I_n} = 0 \quad \forall \vv_2 \in [\mathcal{P}^r_n(I_n)]^d \text{ and }\forall n=1,\dots,N.
	\end{equation*}
	Therefore $\hat{\ww}'=\textbf{0}$ on $[0,T]$. In addition, from \eqref{Eq:DGSpace} we get $\textbf{0}=\ww(t_1^-)=\hat{\ww}(t_1^+)$ that combined with the previous result gives $\hat{\ww}=\textbf{0}$ on $[0,T]$.
	
	Now, since $\hat{\zz}\in \mathcal{V}^{\rr}_{CG}$, we have $\hat{\uu}' = \hat{\ww} = \textbf{0}$ on $[0,T]$. Therefore using again \eqref{Eq:DGSpace} we get
	\begin{equation*}
		(\uu',\vv_1)_{I_n} + [\uu]_{n-1}\cdot \vv_1(t_{n-1}^+)= 0 \quad \forall \vv_1 \in [\mathcal{P}^r_n(I_n)]^d \text{ and }\forall n=1,\dots,N.
	\end{equation*}
	Take $n = N$, then $[\uu]_{N-1}=\textbf{0}$ (from $|||\zz||| = 0$) and therefore $\uu'=\textbf{0}$ on $I_N$. Combining this result with $\uu(T^-)=\textbf{0}$ we get $\uu=\textbf{0}$ on $I_N$ from which we derive $\textbf{0}=\uu(t_{N-1}^+)=\uu(t_{N-1}^-)$. Iterating until $n=2$ we get $\uu=\textbf{0}$ on $I_n$, for any $n=2,\dots,N$. Moreover 
	\begin{equation*}
		\textbf{0}=\uu(t_1^+)=\uu(t_1^-)=\hat{\uu}(t_1^+)=\hat{\uu}(t_1^-)=\hat{\uu}(0^+)=\uu(0^-),
	\end{equation*} since $\hat{\uu}' = \textbf{0}$ on $I_1$. Using again $|||\zz|||=0$ we get $\uu(0^+)=\textbf{0}$, hence $[\uu]_0=\textbf{0}$. Taking $n=1$ we get $\uu=\textbf{0}$ on $I_1$. Thus, $\zz=\textbf{0}$ on $[0,T]$.
\end{proof}

The following result states  the well-posedness of \eqref{Eq:WeakProblem} 
\begin{myprop}
	\label{Pr:WellPosedness} Problem~\eqref{Eq:WeakProblem} admits a unique solution $\uu_{DG} \in \mathcal{V}_{DG}^{\rr}$. 
\end{myprop}
\begin{proof}
	By taking $\vv = \zz$ we get
	\begin{equation*}
		\mathcal{A}(\zz,\zz) = \sum_{n=1}^N (\widetilde{K}\dot{\zz,}\zz)_{I_n} + (A\zz,\zz)_{I_n} + \sum_{n=1}^{N-1} \widetilde{K}[\zz]_n\cdot\zz(t_n^+) + (\widetilde{K}^{\frac{1}{2}}\zz)^2.
	\end{equation*}
	Since $\widetilde{K}$ is symmetric, integrating by parts we have that
	\begin{equation*}
		(\widetilde{K}\dot{\zz},\zz)_{I_n} = \frac{1}{2}\langle \widetilde{K}\zz,\zz \rangle_{t_n^-} - \frac{1}{2}\langle \widetilde{K}\zz,\zz \rangle_{t_{n-1}^+}.
	\end{equation*}
	Then, the second term can be rewritten as
	\begin{equation*}
	(A\zz,\zz)_{I_n} = (-K\ww,\uu)_{I_n} + (K\uu,\ww)_{I_n} + (L\ww,\ww)_{I_n} = ||\widetilde{L}\zz||_{I_n}^2,
	\end{equation*}
	cf. also \eqref{def:KA}. Therefore
	\begin{equation*}
		\mathcal{A}(\zz,\zz) = \sum_{n=1}^N ||\widetilde{L}\zz||_{I_n}^2 + (\widetilde{K}^{\frac{1}{2}}\zz(0^+))^2 + \frac{1}{2}\sum_{n=1}^{N-1} (\widetilde{K}^{\frac{1}{2}}[\zz]_n)^2 + (\widetilde{K}^{\frac{1}{2}}\zz(T^-))^2 = |||\zz|||^2.
	\end{equation*}
	The result follows from Proposition~\ref {Pr:Norm}, the bilinearity of $\mathcal{A}$ and the linearity of $\mathcal{F}$.
\end{proof}

\section{Convergence analysis}\label{Sc:Convergence}

In this section, we first present an \textit{a-priori} stability bound for the numerical solution of \eqref{Eq:WeakProblem} that can be easily obtained by direct application of the Cauchy-Schwarz inequality. Then, we use the latter to prove optimal error estimate for the numerical error, in the energy norm \eqref{Eq:Norm}.

\begin{myprop}
	Let $\ff \in \LL^2(0,T]$, $\hat{\uu}_0, \hat{\uu}_1 \in \mathbb{R}^d$, and let $\zz_{DG} \in \mathcal{V}_{DG}^{\rr}$ be the solution of \eqref{Eq:WeakProblem}, then it holds
	\begin{equation}
		\label{Eq:Stability}
		|||\zz_{DG}||| \lesssim \Big(\sum_{n=1}^N ||L^{-\frac{1}{2}}\ff||_{\LL^(0,T)}^2+(K^{\frac{1}{2}}\hat{\uu}_0)^2+(P^{\frac{1}{2}}\hat{\uu}_1)^2\Big)^{\frac{1}{2}}.
	\end{equation}
\end{myprop}
\begin{proof}
	From the definition of the norm $|||\cdot|||$ given in \eqref{Eq:Norm} and the arithmetic-geometric inequality we have
	\begin{equation*}
	\begin{split}
		|||\zz_{DG}|||^2 &= \mathcal{A}(\zz_{DG},\zz_{DG}) = \mathcal{F}(\zz_{DG}) = \sum_{n=1}^N (\FF,\zz_{DG})_{I_n} + \widetilde{K}\zz_0\cdot\zz_{DG}(0^+) \\
		&\lesssim \frac{1}{2}\sum_{n=1}^N ||L^{-\frac{1}{2}}\ff||_{\LL^2(I_n)}^2 + \frac{1}{2}\sum_{n=1}^N ||\widetilde{L}\zz_{DG}||_{\LL^2(I_n)}^2 + (\widetilde{K}^{\frac{1}{2}} \zz_{0})^2 + \frac{1}{4}(\widetilde{K}^{\frac{1}{2}} \zz_{DG})^2 \\
		&\lesssim \frac{1}{2}\sum_{n=1}^N ||L^{-\frac{1}{2}}\ff||_{\LL^2(I_n)}^2 + (\widetilde{K}^{\frac{1}{2}} \zz_{0})^2 + \frac{1}{2}|||\zz_{DG}|||^2.
	\end{split}
	\end{equation*}
	Hence,
	\begin{equation*}
		|||\zz_{DG}|||^2 \lesssim \sum_{n=1}^N ||L^{-\frac{1}{2}}\ff||_{\LL^2(I_n)}^2 + (K^{\frac{1}{2}} \hat{\uu}_{0})^2 + (P^{\frac{1}{2}}\hat{\uu}_{1})^2.
	\end{equation*}
\end{proof}

Before  deriving  an a priori estimate for the numerical error we introduce some preliminary results. We refer the interested reader to \cite{ScSc2000} for further details.

\begin{mylemma}
	\label{Le:Projector}
	Let $I=(-1,1)$ and $u\in L^2(I)$ continuous at $t=1$, the projector $\Pi^r u \in \mathcal{P}^r(I)$, $r\in\mathbb{N}_0$, defined by the $r+1$ conditions
	\begin{equation}
		\label{Eq:Projector}
		\Pi^r u (1) = u(1), \qquad (\Pi^r u,q)_{I} = 0 \quad\forall\, q\in\mathcal{P}^{r-1}(I),
	\end{equation} 
	is well posed. Moreover, let $I=(a,b)$, $\Delta t = b-a$, $r\in\mathbb{N}_0$ and $u\in H^{s_0+1}(I)$ for some $s_0\in\mathbb{N}_0$. Then
	\begin{equation}
		\label{Eq:ProjectionError}
		||u-\Pi^r u||_{L^2(I)}^2 \le C\bigg(\frac{\Delta t}{2}\bigg)^{2(s+1)}\frac{1}{r^2}\frac{(r-s)!}{(r+s)!}||u^{(s+1)}||_{L^2(I)}^2
	\end{equation}
	for any integer $0\le s \le \min(r,s_0)$. C depends on $s_0$ but it is independent from $r$ and $\Delta t$. 
\end{mylemma}

%
%
%

Proceeding similarly to \cite{ScSc2000}, we now prove the following preliminary estimate for the derivative of the projection $\Pi^r u$.
\begin{mylemma}
	\label{Le:DerivativeProjectionErrorInf}
	Let $u\in H^1(I)$ be continuous at $t=1$. Then, it holds
	\begin{equation}
		\label{Eq:DerivativeProjectionErrorInf}
		||u'-\big(\Pi^r u\big)'||_{L^2(I)}^2 \le C(r+1)\inf_{q \in \mathcal{P}^r(I)} \Bigg\{||u'-q'||_{L^2(I)}^2 \Bigg\}.
	\end{equation}
\end{mylemma}
\begin{proof}
	Let $u' =\sum_{i=1}^{\infty} u_i L'_i$ be the Legendre expansion of $u'$ with coefficients $u_i\in\mathbb{R}$, $i=1,\dots,\infty$. Then (cfr. Lemma 3.2 in \cite{ScSc2000})
	\begin{equation*}
		\big(\Pi^r u\big)'=\sum_{i=1}^{r-1} u_i L'_i + \sum_{i=r}^{\infty} u_i L'_r
	\end{equation*}
	Now, for $r\in\mathbb{N}_0$, we denote by $\widehat{P}^r$ the $L^2(I)$-projection onto $\mathcal{P}^r(I)$. Hence,
	\begin{equation*}
		u' - \big(\Pi^r u\big)'= \sum_{i=r}^{\infty} u_i L'_i - \sum_{i=r}^{\infty} u_i L'_r = \sum_{i=r+1}^{\infty} u_i L'_i - \sum_{i=r+1}^{\infty} u_i L'_r = u' - \big(\widehat{P}^r u\big)' - \sum_{i=r+1}^{\infty} u_i L'_r.
	\end{equation*}
	Recalling that $||L'_r||_{L^2(I)} = r(r+2)$ we have
	\begin{equation*}
		||u' - \big(\Pi^r u\big)'||_{L^2(I)}^2 \le ||u' - \big(\widehat{P}^r u\big)'||_{L^2(I)}^2  - \Bigg|\sum_{i=r+1}^{\infty} u_i\Bigg| r(r+1).
	\end{equation*}
	Finally, we use that  
$		\Bigg|\sum_{i=r+1}^{\infty} u_i\Bigg| \le \frac{C}{r}||u'||_{L^2(I)}
$	 (cfr. Lemma~3.6 in \cite{ScSc2000}) and get
	\begin{equation}
	 	\label{Eq:DerivativeProjectionError} 
		||u'-\big(\Pi^r u\big)'||_{L^2(I)}^2 \le C\big\{||u'-\big(\widehat{P}^r u\big)'||_{L^2(I)}^2+(r+1)||u'||_{L^2(I)}^2 \big\}.
	\end{equation}
	Now consider $q\in\mathcal{P}^r(I)$ arbitrary and insert $u'-q'$ into \eqref{Eq:DerivativeProjectionError}. The thesis follows from the reproducing properties of projectors $\Pi^r u$ and $\widehat{P}^r u$ and from the fact that $||u-\widehat{P}^r u||_{L^2(I)} \le ||u-q||_{L^2(I)} $ for any $q\in\mathcal{P}^r(I)$.
\end{proof}

By employing Proposition~3.9 in \cite{ScSc2000} and Lemma \ref{Le:DerivativeProjectionErrorInf} we obtain the following result.
\begin{mylemma}
	\label{Le:DerivativeProjectionError}
	Let $I=(a,b)$, $\Delta t = b-a$, $r\in\mathbb{N}_0$ and $u\in H^{s_0+1}(I)$ for some $s_0\in\mathbb{N}_0$. Then
	\begin{equation*}
	||u'-\big(\Pi^r u\big)'||_{L^2(I)}^2 \lesssim \bigg(\frac{\Delta t}{2}\bigg)^{2(s+1)}(r+2)\frac{(r-s)!}{(r+s)!}||u^{(s+1)}||_{L^2(I)}^2
	\end{equation*}
	for any integer $0\le s \le \min(r,s_0)$. The hidden constants depend on $s_0$ but are independent from $r$ and $\Delta t$. 
\end{mylemma}

Finally we observe that the bilinear form appearing in formulation \eqref{Eq:WeakProblem} is strongly consistent, i.e.
\begin{equation}
	\label{Eq:Consistency}
	\mathcal{A}(\zz-\zz_{DG},\vv) = 0 \qquad \forall\,\vv\in\mathcal{V}^{\rr}_{DG}.
\end{equation}
We now state the following convergence result.
\begin{myth}
	\label{Th:ErrorEstimate}
	Let $\hat{\uu}_{0},\hat{\uu}_{1} \in \mathbb{R}^{d}$. Let $\zz$ be the solution of problem~\eqref{Eq:FirstOrderSystem2} and let $\zz_{DG}\in\mathcal{V}_{DG}^{\rr}$ be its finite element approximation. If $\zz|_{I_n}\in \HH^{s_n}(I_n)$, for any $n=1,\dots,N$ with $s_n\geq2$, then it holds
	\begin{equation}
		\label{Eq:ErrorEstimate}
		|||\zz-\zz_{DG}||| \lesssim \sum_{n=1}^N \bigg(\frac{\Delta t}{2}\bigg)^{\mu_n+\frac{1}{2}}\Bigg((r_n+2)\frac{(r_n-\mu_n)!}{(r_n+\mu_n)!}\Bigg)^{\frac{1}{2}}||\zz||_{H^{\mu_n+1}(I_n)},
	\end{equation}
	where $\mu_n = \min(r_n,s_n)$, for any $n=1,\dots,N$ and the hidden constants depend on the norm of matrices $L$, $K$ and $A$.
\end{myth}
\begin{proof}
	We set $\ee = \zz - \zz_{DG} = (\zz - \Pi_I^r \zz) + (\Pi_I^r \zz - \zz_{DG}) = \ee^{\pi} + \ee^{h}$. Hence we have $|||\ee||| \le |||\ee^{\pi}||| + |||\ee^{h}|||$. Employing the properties of the projector \eqref{Eq:Projector} and estimates \eqref{Eq:ProjectionError} and \eqref{Eq:DerivativeProjectionError}, we can bound $|||\ee^{\pi}|||$ as
	\begin{equation*}
	\begin{split}
		|||\ee^{\pi}|||^2 &= \sum_{n=1}^N ||\widetilde{L}\ee^{\pi}||_{L^2(I_n)}^2 + \frac{1}{2}(\widetilde{K}^{\frac{1}{2}}\ee^{\pi}(0^+))^2 + \frac{1}{2}\sum_{n=1}^{N-1}(\widetilde{K}^{\frac{1}{2}}[\ee^{\pi}]_n)^2 + \frac{1}{2}(\widetilde{K}^{\frac{1}{2}}\ee^{\pi}(T^-))^2 \\
		& = \sum_{n=1}^N ||\widetilde{L}\ee^{\pi}||_{L^2(I_n)}^2 + \frac{1}{2} \sum_{n=1}^N \Bigg(-\int_{t_{n-1}}^{t_{n}}\widetilde{K}^{\frac{1}{2}}\dot{\ee}^{\pi}(s)ds\Bigg)^2 \\
		& \lesssim \sum_{n=1}^N \Big(||\ee^{\pi}||_{L^2(I_n)}^2 + \Delta t ||\dot{\ee^{\pi}}||_{L^2(I_n)}^2 \Big) \\
		& \lesssim \sum_{n=1}^N \bigg[\bigg(\frac{\Delta t_n}{2}\bigg)^{2\mu_n+2} \frac{1}{r_n^2} + \bigg(\frac{\Delta t_n}{2}\bigg)^{2\mu_n+1} (r_n+2)\bigg] \frac{(r_n-\mu_n)!}{(r_n+\mu_n)!}||\zz||_{H^{\mu_n+1}(I_n)} \\
		& \lesssim \sum_{n=1}^N \bigg(\frac{\Delta t_n}{2}\bigg)^{2\mu_n+1} (r_n+2) \frac{(r_n-\mu_n)!}{(r_n+\mu_n)!}||\zz||_{H^{\mu_n+1}(I_n)},
	\end{split}
	\end{equation*}
	where $\mu_n = \min(r_n,s_n)$, for any $n=1,\dots,N$.
	For the term $|||\ee_{h}|||$ we use \eqref{Eq:Consistency} and integrate by parts to get
	\begin{equation*}
	\begin{split}
		|||\ee^{h}|||^2 &= \mathcal{A}(\ee^h,\ee^h) = -\mathcal{A}(\ee^{\pi},\ee^h) \\ 
		& = \sum_{n=1}^N (\widetilde{K}\dot{\ee}^{\pi},\ee^h)_{I_n} + \sum_{n=1}^N(A\ee^{\pi},\ee^h)_{I_n} + \sum_{n=1}^{N-1} \widetilde{K}[\ee^{\pi}]_n\cdot\ee^h(t_n^+) + \widetilde{K}\ee^{\pi}(0^+)\cdot\ee^h(0^+) \\
		& = \sum_{n=1}^N (\widetilde{K}\ee^{\pi},\dot{\ee}^h)_{I_n} + \sum_{n=1}^N(A\ee^{\pi},\ee^h)_{I_n} + \sum_{n=1}^{N-1} \widetilde{K}[\ee^{h}]_n\cdot\ee^{\pi}(t_n^-) - \widetilde{K}\ee^{\pi}(T^-)\cdot\ee^h (T^-).
	\end{split}
	\end{equation*}
	Thanks to  \eqref{Eq:Projector}, only the second term of the last equation above does not vanish. Thus, we employ the Cauchy-Schwarz and arithmetic-geometric inequalities to obtain
	\begin{equation*}
		|||\ee^{h}|||^2 = \sum_{n=1}^N(A\ee^{\pi},\ee^h)_{I_n} \lesssim \frac{1}{2} \sum_{n=1}^N ||\ee^{\pi}||_{L^2(I_n)}^2 + \frac{1}{2} \sum_{n=1}^N ||\widetilde{L}\ee^{h}||_{L^2(I_n)}^2 \lesssim \frac{1}{2} \sum_{n=1}^N ||\ee^{\pi}||_{L^2(I_n)}^2 + \frac{1}{2}|||\ee^h|||^2.
 	\end{equation*}
 	Hence,
 	\begin{equation*}
 		|||\ee^{h}|||^2 \lesssim \sum_{n=1}^N \bigg(\frac{\Delta t_n}{2}\bigg)^{2\mu_n+2} \frac{1}{r_n^2} \frac{(r_n-\mu_n)!}{(r_n+\mu_n)!}||\zz||_{H^{\mu_n+1}(I_n)},
 	\end{equation*}
	where $\mu_n = \min(r_n,s_n)$, for any $n=1,\dots,N$ and the thesis follows.
\end{proof}

\section{Algebraic formulation}
\label{Sc:AlgebraicFormulation}
In this section we derive the algebraic formulation stemming after DG discretization of \eqref{Eq:WeakProblem} for the time slab $I_n$. 
We consider on $I_n$ a local polynomial degree $r_n$. In practice, since we use discontinuous functions, we can compute the numerical solution  one time slab at time, assuming the initial conditions stemming from the previous time slab known. Hence, problem \eqref{Eq:WeakProblem} reduces to: find $\zz\in V^{r_n}(I_n)$ such that
\begin{equation}
	\label{Eq:WeakFormulationReduced}
	(\widetilde{K}\dot{\zz},\vv)_{I_n} + (A\zz,\vv)_{I_n} + \langle\widetilde{K}\zz,\vv\rangle_{t_{n-1}^+} = (\FF,\vv)_{I_n} + \widetilde{K}\zz(t_{n-1}^-)\cdot\vv({t_{n-1}^+}), \quad \forall\,n=1,\dots,N.
\end{equation}

Introducing a basis $\{\psi^{\ell}(t)\}_{{\ell}=1,\dots,r_n+1}$ for the polynomial space $\mathbb{P}^{r_n}(I_n)$ we define a vectorial basis $\{ \boldsymbol{\Psi}_i^{\ell}(t) \}_{i=1,\dots,2d}^{{\ell}=1,\dots,r_n+1}$ of $V_n^{r_n}$ where
\begin{equation*}
	\{ \boldsymbol{\Psi}_i^{\ell}(t) \}_j = 
	\begin{cases}
		\psi^{\ell}(t) & {\ell} = 1,\dots,r_n+1, \quad \text{if } i=j, \\
		0 & {\ell} = 1,\dots,r_n+1, \quad \text{if } i\ne j.
	\end{cases}
\end{equation*}
Then, we set $D_n=d(r_n+1)$ and write the trial function $\zz_n = \zz_{DG}|_{I_n} \in V_n^{r_n}$ as
\begin{equation*}
	\zz_n(t) = \sum_{j=1}^{2d} \sum_{m=1}^{r_n+1} \alpha_{j}^m \boldsymbol{\Psi}_j^m(t),
\end{equation*}
where $\alpha_{j}^m\in\mathbb{R}$ for $j=1,\dots,2d$, $m=1,\dots,r_n+1$. Writing \eqref{Eq:WeakFormulationReduced} for any test function $\boldsymbol{\Psi}_i^{\ell}(t)$, $i=1,\dots,2d$, $\ell=1\,\dots,r_n+1$ we obtain the linear system
\begin{equation}
	\label{Eq:LinearSystem}
	M\ZZ_n = \GG_n,
\end{equation}
where $\ZZ_n,\GG_n \in \mathbb{R}^{2D_n}$ are the vectors of expansion coefficient corresponding to the numerical solution and the right hand side on the interval $I_n$ by the chosen basis. Here $M\in\mathbb{R}^{2D_n\times2D_n}$ is the local stiffness matrix defined as
\begin{equation}
	\label{Eq:StiffnessMatrix}
	M = \widetilde{K} \otimes (N^1+N^3) + A \otimes N^2
	= \begin{bmatrix}
		K \otimes (N^1 + N^3) & -K \otimes N^2 \\
		K \otimes N^2 & P \otimes (N^1+N^3) + L \otimes N^2 
	\end{bmatrix},
\end{equation}
where $N^1,N^2,N^3 \in \mathbb{R}^{r_n+1}$ are the local time matrices
\begin{equation}
	\label{Eq:TimeMatrices}
	N_{{\ell}m}^1 = (\dot{\psi}^m,\psi^{\ell})_{I_n}, \qquad N_{{\ell}m}^2 = (\psi^m,\psi^{\ell})_{I_n}, \qquad N_{{\ell}m}^3 = \langle\psi^m,\psi^{\ell}\rangle_{t_{n-1}^+},
\end{equation}
for $\ell,m=1,...,r_n+1$. Similarly to \cite{ThHe2005}, we reformulate system \eqref{Eq:LinearSystem} to reduce the computational cost of its resolution phase. We first introduce the vectors $\GG_n^u,\, \GG_n^w,\, \UU_n,\, \WW_n \in \mathbb{R}^{D_n}$ such that
\begin{equation*}
	\GG_n = \big[\GG_n^u, \GG_n^w\big]^T, \qquad \ZZ_n = \big[\UU_n, \WW_n\big]^T
\end{equation*}
and the matrices
\begin{equation}
N^4 = (N^1+N^3)^{-1}, \qquad N^5 = N^4N^2, \qquad N^6 = N^2N^4, \qquad N^7 = N^2N^4N^2.
\end{equation}
Next, we apply a block Gaussian elimination getting
\begin{equation*}
	M = \begin{bmatrix}
	K \otimes (N^1 + N^3) & -K \otimes N^2 \\
	0 & P \otimes (N^1+N^3) + L \otimes N^2 + K \otimes N^7
	\end{bmatrix},
\end{equation*}
and
\begin{equation*}
	\GG_n = \begin{bmatrix}
	\GG_n^u \\
	\GG_n^w - \mathcal{I}_d\otimes N^6 \GG_n^u
	\end{bmatrix}.
\end{equation*}

We define the matrix $\widehat{M}_n\in\mathbb{R}^{D_n\times D_n}$ as 
\begin{equation}\label{Eq:TimeMatrix}
	\widehat{M}_n = P \otimes (N^1+N^3) + L \otimes N^2 + K \otimes N^7,
\end{equation}
and the vector $\widehat{\GG}_n\in\mathbb{R}^{D}$ as 
\begin{equation}
	\widehat{\GG}_n = \GG_n^w - \mathcal{I}_{d}\otimes N^6 \GG_n^u.
\end{equation}

Then, we multiply the first block by $K^{-1}\otimes N^4$ and, exploiting the properties of the Kronecker product, we get 
\begin{equation*}
	\begin{bmatrix}
	\mathcal{I}_{D_n} & -\mathcal{I}_{d} \otimes N^5 \\
	0 & \widehat{M}_n
	\end{bmatrix} 
	\begin{bmatrix}
	\UU_n \\
	\WW_n
	\end{bmatrix} = 
	\begin{bmatrix}
	(K^{-1}\otimes N^4)\GG_n^u \\
	\widehat{\GG}_n
	\end{bmatrix}.
\end{equation*}

Therefore, we first obtain the velocity $\WW_n$ by solving the linear system
\begin{equation}\label{Eq:VelocitySystem}
	\widehat{M}_n \WW_n = \widehat{\GG}_n,
\end{equation}	 
and then, we can compute the displacement $\UU_n$ as
\begin{equation}\label{Eq:DisplacementUpdate1}
	\UU_n = \mathcal{I}_{d} \otimes N^5 \WW_n + (K^{-1}\otimes N^4)\GG_n^u.
\end{equation}	

Finally, since $\big[\GG_n^u\big]_i^{\ell} = K\UU(t_{n-1}^-)\cdot\boldsymbol{\Psi}_i^{\ell}(t_{n-1}^+)$, by defining $\bar{\GG}_n^u\in \mathbb{R}^{D_n}$ as
\begin{equation}
	\big[\bar{\GG}_n^u\big]_i^{\ell} = \UU(t_{n-1}^-)\cdot\boldsymbol{\Psi}_i^{\ell}(t_{n-1}^+),
\end{equation}
we can rewrite \eqref{Eq:DisplacementUpdate1} as 
\begin{equation}\label{Eq:AltDisplacementUpdate2}
	\UU_n = \mathcal{I}_{d} \otimes N^5 \WW_n + (\mathcal{I}_{d}\otimes N^4)\bar{\GG}_n^u.
\end{equation}

\section{Numerical results}
\label{Sc:NumericalResults}
In this section we report a wide set of numerical experiments to validate the theoretical estimates and asses the performance of the DG method proposed in Section \ref{Sc:Method}. We first present a set of verification tests for scalar- and vector-valued problems, then we test our formulation onto two- and three-dimensional elastodynamics wave propagation problems, through the open source software SPEED (\url{http://speed.mox.polimi.it/}).


\subsection{Scalar problem}
\label{Sec:1DConvergence}
For a time interval $I=[0,T]$, with $T=10$, we solve the scalar problem
\begin{equation}
	\label{Eq:ScalarProblem}
	\begin{cases}
       \dot{u}(t) = w(t)	& \forall t\in [0,10],\\
	   \dot{w}(t) + 5 w(t) + 6u(t) = f(t) & \forall t\in [0,10], \\
		u(0) = 2, \\
		w(0) = -5,
	\end{cases}
\end{equation}
whose exact solution is $ \zz(t) = (w(t),u(t)) = (-3e^{-3t}-3e^{-2t},e^{-3t}+e^{-2t})$ for $t\in[0,10]$.

We partition the time domain $I$ into $N$ time slabs of uniform length $\Delta t$ and we suppose the polynomial degree to be constant for each time-slab, i.e. $r_n =  r$, for any $n=1,\dots,N$.  We first compute 
the error $|||\zz_{DG} -\zz |||$ as a function of the time-step $\Delta t$ for several choices of the polynomial degree $r$, as shown in Figure \ref{Fig:ConvergenceTest0D} (left). The obtained results confirms the super-optimal convergence properties of the scheme as shown in \eqref{Eq:ErrorEstimate}. Finally, since $\zz \in C^{\infty}(\mathbb{R})$, from Figure \ref{Fig:ConvergenceTest0D} (right) we can observe that the numerical error decreases exponentially with respect to the polynomial degree $r$.

\begin{figure}[htbp]
\centering
\includegraphics[width=0.49\textwidth]{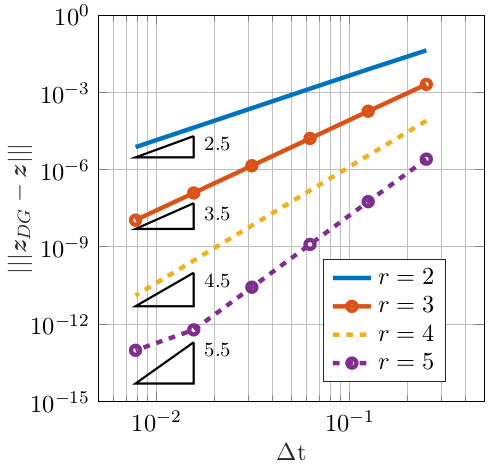}
\includegraphics[width=0.49\textwidth]{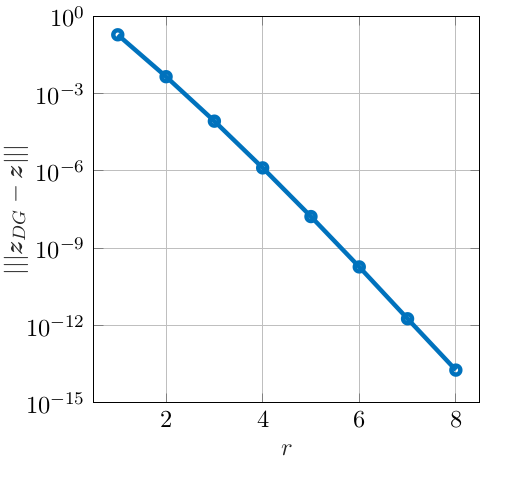}
	\caption{Test case of Section~\ref{Sec:1DConvergence}. Left: computed error $|||\zz_{DG}-\zz|||$ as a function of time-step $\Delta t$, with $r = 2,3,4,5$. Right: computed error $|||\zz-\zz_{DG}|||$ as a function of polynomial degree $r$, using a time step $\Delta t = 0.1$.}\label{Fig:ConvergenceTest0D}

\end{figure}




\subsection{Application to a the visco-elastodynamics system}
\label{Sec:AppVE}
In the following experiments we employ the proposed DG method to solve the second-order differential system of equations stemming from the spatial discretization of the visco-elastodynamics equation:

\begin{equation}
\label{Eq:Elastodynamic}
\begin{cases}
\partial_t \bu - \bw = \textbf{0},  	& \text{in } \Omega\times(0,T],\\
\rho\partial_{t}\bw + 2\rho\zeta\bw + \rho \zeta^2\bu - \nabla\cdot\bsig(\bu) = \textbf{f}, & \text{in } \Omega\times(0,T],\\
\end{cases}
\end{equation}
where  $\Omega\in\mathbb{R}^\mathsf{d}$, $\mathsf{d}=2,3$, is an open bounded polygonal domain. Here, $\rho$ represents the density of the medium, $\zeta$ is a decay factor whose dimension is inverse of time, $\textbf{f}$ is a given source term (e.g. seismic source) and $\bsig$ is the stress tensor encoding the Hooke's law 
\begin{equation}
\bsig(\bu)_{ij} = \lambda\sum_{k=1}^\mathsf{d} \frac{\partial u_k}{\partial x_k} + \mu \left( \frac{\partial u_i}{\partial x_j} + \frac{\partial u_j}{\partial x_i} \right), \quad {\rm for} \; i,j=1,...,\mathsf{d},
\end{equation}
being $\lambda$ and $\mu$ the first and the second Lam\'e parameters, respectively. Problem \eqref{Eq:Elastodynamic} is usually supplemented with boundary conditions for $\bu$ and initial conditions for $\bu$ and $\bw$, that we do not report here for brevity.  
Finally, we suppose problem's data are regular enough to gaurantee its well-posedness \cite{AntoniettiFerroniMazzieriQuarteroni_2017}.

By employing a finite element discretization (either in its continuous or discontinuous variant) for the semi-discrete approximation (in space) of \eqref{Eq:Elastodynamic} we obtain the following system
\begin{equation*}
\left( \begin{matrix}
I & 0 \\
0 & P
\end{matrix} \right)\left( \begin{matrix}
\dot{\uu} \\
\dot{\ww}
\end{matrix} \right)  + \left( \begin{matrix}
0 & -I \\
K & L
\end{matrix} \right)\left( \begin{matrix}
{\uu} \\
{\ww}
\end{matrix} \right) = \left( \begin{matrix}
\textbf{0} \\
\ff
\end{matrix} \right), 
\end{equation*}
that can be easily rewritten as in \eqref{Eq:FirstOrderSystem1}.
We remark that within the matrices and the right hand side are encoded the boundary conditions associated to \eqref{Eq:Elastodynamic}.
%
%
%
%
%
%
%
%
%
%
For the space discretization of \eqref{Eq:Elastodynamic}, we consider in the following a high order Discontinuous Galerkin method based either on general polygonal meshes (in two dimensions) \cite{AnMa2018} or on unstructured hexahedral meshes (in three dimensions) \cite{mazzieri2013speed}. 

For the forthcoming experiments we denote by $h$ the granularity of the spatial mesh and $p$ the order of polynomials employed for space approximation. The combination of space and time DG methods  yields to a high order space-time DG method that we denote by STDG. 

Remark that the latter has been implemented in the open source software SPEED (\url{http://speed.mox.polimi.it/}). 

\subsubsection{A two-dimensional test case with space-time polyhedral meshes}
\label{Sec:2DConvergence}
As a first verification test we consider problem~\eqref{Eq:Elastodynamic} in a bidimensional setting, i.e. $\Omega = (0,1)^2 \subset \mathbb{R}^2$. 
We set the mass density $\rho=1$, the Lamé coefficients $\lambda=\mu=1$, $\zeta = 1$ and choose the data $\textbf{f}$ and the initial conditions  such that the exact solution of \eqref{Eq:Elastodynamic} is $\textbf{z} = (\bu,\bw)$ where
\begin{equation*}
\bu = e^{-t}
\begin{bmatrix}
	-\sin^2(\pi x)\sin(2\pi y) \\
	\sin(2\pi x)\sin^2(\pi y)
\end{bmatrix}, \qquad \bw = \partial_t\bu. 
\end{equation*}
We consider a polygonal mesh (see Figure~\ref{fig:dgpolyspace-time}) made by 60 elements and set $p=8$. We take $T=0.4$ and divide the temporal iterval $(0,T]$ into $N$ time-slabs of uniform lenght $\Delta t$.


\begin{figure}[h!]
	\centering
	\includegraphics[width=0.5\textwidth]{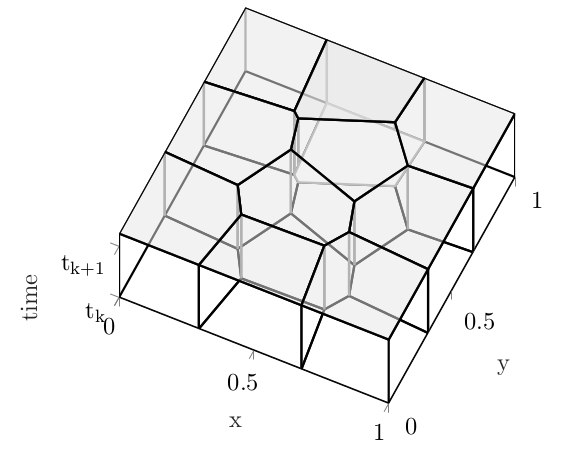}
	\caption{Test case of Section~\ref{Sec:2DConvergence}. Example of space-time polygonal grid used for the verification test.}
	\label{fig:dgpolyspace-time}
\end{figure}

In Figure~\ref{Fig:ConvergenceTest2D} (left) we show the energy norm \eqref{Eq:Norm} of the numerical error $|||\zz_{DG}-\zz|||$ computed for several choices of time polynomial degree $r=1,2,3$ by varying the time step $\Delta t$. We can observe that the error estimate \eqref{Eq:ErrorEstimate} is confirmed by our numerical results. Moreover, from Figure~\ref{Fig:ConvergenceTest2D} (right) we can observe that the numerical error decreases exponentially with respect to the polynomial degree $r$. In the latter case we fixed $\Delta t = 0.1$ and use 10 polygonal elements for the space mesh, cf. Figure~\ref{fig:dgpolyspace-time}.

\begin{figure}[h!]
   \includegraphics[width=0.49\textwidth]{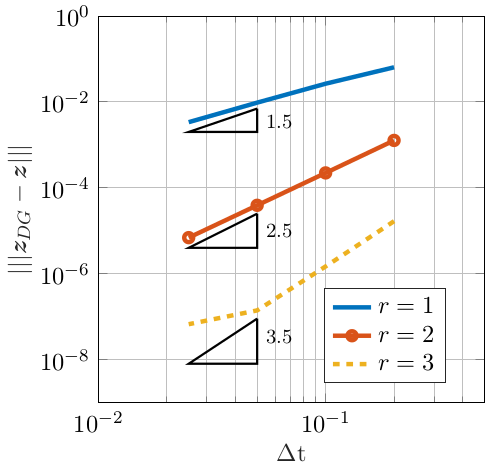}
\includegraphics[width=0.49\textwidth]{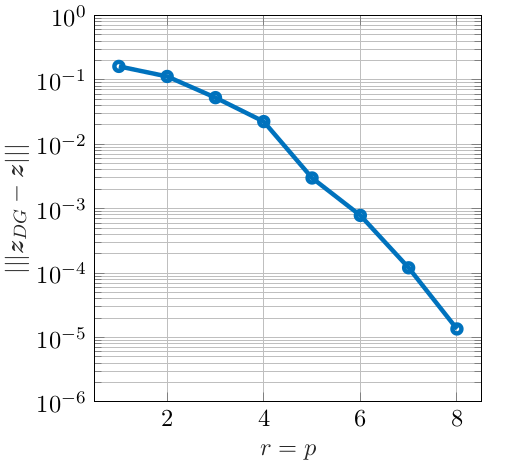}
      \caption{Test case of Section~\ref{Sec:2DConvergence}. Left: computed error $|||\zz-\zz_{DG}|||$ as a function of time-step $\Delta t$ for  $r = 1,2,3$, using a space discretization with a polygonal mesh composed of $60$ elements and $p=8$. Right: computed error $|||\zz-\zz_{DG}|||$ as a function of the polynomial degree $r=p$, using a spatial grid composed of 10 elements and a time step $\Delta t = 0.1$. }
	\label{Fig:ConvergenceTest2D}
\end{figure}

%
%

\subsubsection{A three-dimensional test case with space-time polytopal meshes}
\label{Sec:3DConvergence}
As a second verification test we consider problem~\eqref{Eq:Elastodynamic} for in a three dimensional setting. Here, we consider $\Omega = (0,1)^3 \subset \mathbb{R}^3$, $T=10$ and we set the external force $\boldsymbol{f}$ and the initial conditions so that the exact solution of \eqref{Eq:Elastodynamic} is $\textbf{z} = (\bu,\bw)$ given by
\begin{equation}
\label{Testcase}
\bu = \cos(3\pi t)
\begin{bmatrix}
\sin(\pi x)^2\sin(2\pi y)\sin(2\pi z) \\
\sin(2\pi x)\sin(\pi y)^2\sin(2\pi z) \\ 				
\sin(2\pi x)\sin(2\pi y)\sin(\pi z)^2
\end{bmatrix}, \quad \bw = -3\pi\cos(3\pi t) \bu.
\end{equation}


We partition $\Omega$ by using a conforming hexahedral mesh of granularity $h$, and we use a uniform time domain partition of step size $\Delta t$ for the time interval $[0,T]$. We choose a polynomial degree $ p \ge 2$ for the space discretization and  $ r \ge 1$ for the temporal one. We firstly set $h=0.0125$ corresponding to $512$ elements and fix $p=6$, and let the time step $\Delta t$ varying from $0.4$ to $0.00625$ for $r=1,2,3,4$. The computed energy errors are shown in Figure \ref{Fig:ConvergenceTest3D} (left). We can observe that the numerical results are in agreement with the theoretical ones, cf. Theorem~\ref{Th:ErrorEstimate}. We note that with $r=4$, the error reaches a plateau for $\Delta t \leq 0.025$. However, this effect could be easily overcome by increasing the spatial polynomial degree $p$ and/or by refining the mesh size $h$.  
Then, we fix a grid  size $h=0.25$, a time step $\Delta t=0.1$ and let vary together the polynomial degrees, $p=r=2,3,4,5$. Figure \ref{Fig:ConvergenceTest3D} (right) shows an exponential decay of the error.

\begin{figure}
\includegraphics[width=0.49\textwidth]{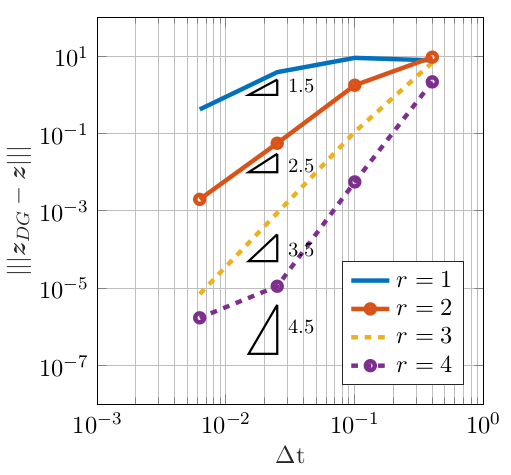}
\includegraphics[width=0.49\textwidth]{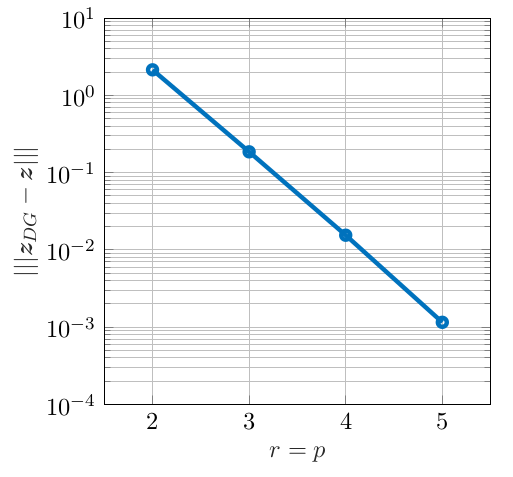}
	\caption{Test case of Section~\ref{Sec:3DConvergence}. Left: computed errors $|||\zz_{DG}-\boldsymbol{u}|||$ as a function of the time-step $\Delta t$, with $r=1,2,3,4$, $h=0.125$ and $p=6$. Right: computed errors $|||\zz_{DG}-\boldsymbol{u}|||$ as a function of the polynomial degree $p=r$, with $\Delta t = 0.1$, $h=0.25$.}
	\label{Fig:ConvergenceTest3D}
\end{figure}

%

\subsubsection{Plane wave propagation}
\label{Sec:PlaneWave}
The aim of this test is to compare the performance of the proposed method STDG with the space-time DG method (here referred to as STDG$_0$) firstly presented in \cite{Paper_Dg-Time} and then applied to 3D problems in \cite{AnMaMi20}. The difference between STDG$_0$ and STDG is in the way the time approximation is obtain. Indeed, the former integrates the second order in time differential problem, whereas the latter discretizes the first order in time differential system.   
On the one hand, as pointed out in \cite{AnMaMi20}, the main limitation of the STDG$_0$ method is the ill-conditioning of the resulting stiffness matrix that makes the use of iterative solvers quite difficult. Hence, for STDG$_0$ direct methods are used forcing to store the stiffness matrix and greatly reducing the range of problems affordable by that method. 
On the other hand, even if the final linear systems stemming from STDG$_0$ and STDG methods are very similar (in fact they only differ upon the definition of the (local) time matrices) we obtain for the latter a well-conditioned system matrix making iterative methods employable and complex 3D problems solvable. 

Here, we consider a plane wave propagating along the vertical direction in two (horizontally stratified) heterogeneous domains. The source plane wave is polarized in the $x$ direction and its time dependency is given by a unit amplitude Ricker wave with peak frequency at $2~{\rm Hz}$. We impose a free surface condition on the top surface, absorbing boundary conditions on the bottom surface and homogeneous Dirichlet conditions along the $y$ and $z$ direction on the remaining boundaries. We solve the problem in two domains that differs from dimensions and material properties, and are called as Domain A and Domain B, respectively.

Domain A has dimension $\Omega=(0,100)~{\rm m}\times(0,100)~{\rm m}\times(-500,0)~{\rm m}$, cf. Figure~\ref{Fig:TutorialDomain}, and is partitioned into 3 subdomains corresponding to the different material layers, cf. Table~\ref{Tab:TutorialMaterials}. The subdomains are discretized in space with a uniform cartesian hexahedral grid of size $h = 50~{\rm m}$ that results in 40 elements. 
Domain B has dimensions $\Omega=(0,100)~{\rm m}\times(0,100)~{\rm m}\times(-1850,0)~{\rm m}$, and has more layers, cf. Figure~\ref{Fig:TorrettaDomain} and Table~\ref{Tab:TorrettaMaterials}. The subdomains are discretized in space with a cartesian hexahedral grid of size $h$ ranging from $15~{\rm m}$ in the top layer to $50~{\rm m}$ in the bottom layer. Hence, the total number of elements is 1225.

\begin{figure}
	\begin{minipage}{\textwidth}
		\begin{minipage}{0.3\textwidth}
			\centering
			\includegraphics[width=0.7\textwidth]{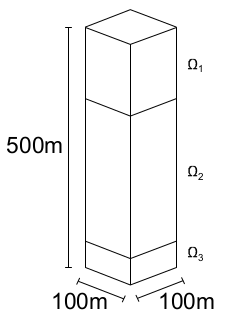}%
			\captionof{figure}{Test case of Section~\ref{Sec:PlaneWave}-Domain A. Computational domain $\Omega = \cup_{\ell=1}^{3}\Omega_{\ell}$.}
			\label{Fig:TutorialDomain}
		\end{minipage}
		\hfill
		\begin{minipage}{0.65\textwidth}
			\centering
			\begin{tabular}{|l|r|r|r|r|r|}
				\hline
				Layer & Height $[m]$ & $\rho [kg/m^3]$ & $c_p [m/s]$ & $c_s [m/s]$ &  $\zeta [1/s]$ \\
				\hline
				\hline
				$\Omega_1$ & $ 150 $ & $1800$ & $600$ & $300$ & $0.166$ \\
				\hline
				$\Omega_2$ & $ 300 $ & $2200$ & $4000$ & $2000$ & $0.025$ \\
				\hline
				$\Omega_3$ & $ 50 $ & $2200$ & $4000$ & $2000$ & $0.025$ \\
				\hline
			\end{tabular}
			\captionof{table}{Mechanical properties for test case of Section~\ref{Sec:PlaneWave}-Domain A. Here, the Lam\'e parameters $\lambda$ and $\mu$ can be obtained through the relations $\mu = \rho c_s^2$ and $\lambda = \rho c_p^2 -\mu$.}
			\label{Tab:TutorialMaterials}
		\end{minipage}
	\end{minipage}
\end{figure}

\begin{figure}
	\begin{minipage}{\textwidth}
		\begin{minipage}{0.3\textwidth}
			\centering
			\includegraphics[width=0.7\textwidth]{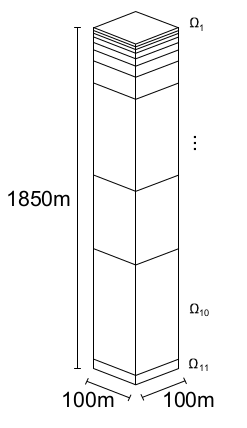}%
			\captionof{figure}{Test case of Section~\ref{Sec:PlaneWave}-Domain B. Computational domain $\Omega = \cup_{\ell=1}^{11}\Omega_{\ell}$.}
			\label{Fig:TorrettaDomain}
		\end{minipage}
		\hfill
		\begin{minipage}{0.65\textwidth}
			\centering
			\begin{tabular}{|l|r|r|r|r|r|}
				\hline
				Layer & Height $[m]$ & $\rho [kg/m^3]$ & $c_p [m/s]$ & $c_s [m/s]$ &  $\zeta [1/s]$ \\
				\hline
				\hline
				$\Omega_1$ & $ 15 $ & $1800$ & $1064$ & $236$ & $0.261$ \\
				\hline
				$\Omega_2$ & $ 15 $ & $1800$ & $1321$ & $294$ & $0.216$ \\
				\hline
				$\Omega_3$ & $ 20 $ & $1800$ & $1494$ & $332$ & $0.190$ \\
				\hline
				$\Omega_4$ & $ 30 $ & $1800$ & $1664$ & $370$ & $0.169$ \\
				\hline
				$\Omega_5$ & $ 40 $ & $1800$ & $1838$ & $408$ & $0.153$ \\
				\hline
				$\Omega_6$ & $60 $ & $1800$ & $2024$ & $450$ & $0.139$ \\
				\hline
				$\Omega_7$ & $ 120 $ & $2050$ & $1988$ & $523$ & $0.120$ \\
				\hline
				$\Omega_8$ & $500  $ & $2050$ & $1920$ & $600$ & $0.105$ \\
				\hline
				$\Omega_9$ & $ 400 $ & $2400$ & $3030$ & $1515$ & $0.041$ \\
				\hline
				$\Omega_{10}$ & $ 600 $ & $2400$ & $4180$ & $2090$ & $0.030$ \\
				\hline
				$\Omega_{11}$ & $ 50 $ & $2450$ & $5100$ & $2850$ & $0.020$ \\
				\hline
			\end{tabular}
			\captionof{table}{Mechanical properties for test case of Section~\ref{Sec:PlaneWave}-Domain B. Here, the Lam\'e parameters $\lambda$ and $\mu$ can be obtained through the relations $\mu = \rho c_s^2$ and $\lambda = \rho c_p^2 -\mu$.}
			\label{Tab:TorrettaMaterials}
		\end{minipage}
	\end{minipage}
\end{figure}

In Figure~\ref{Fig:PlanewaveDisplacement} on the left (resp. on the right) we report the computed displacement $\uu$ along the $x-$axis, registered at point  $P=(50, 50, 0)~{\rm m}$  located on the top surface 
for Domain A (resp. Domain B). 
We compare the results with those obtained in   \cite{AnMaMi20}, choosing a polynomial degree $p=r=2$ in both space and time variables and a time step $\Delta t = 0.01$. In both cases, we can observe a perfect agreement of the two solutions.

\begin{figure}
\includegraphics[width=0.49\textwidth]{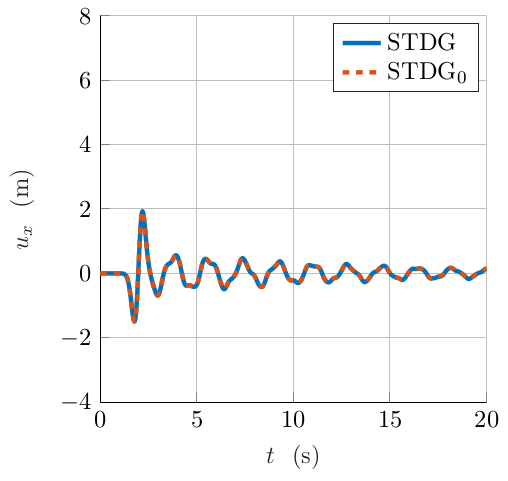}
\includegraphics[width=0.49\textwidth]{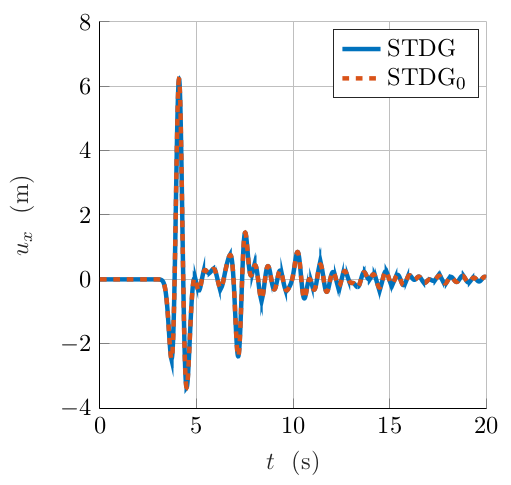}
	\caption{Test case of Section~\ref{Sec:PlaneWave}. Computed displacement $\uu$ along $x-$axis registered at $P=(50, 50, 0)~{\rm m}$ obtained employing the proposed formulation, i.e. STDG method,  and the  method \cite{AnMaMi20}, i.e. STDG$_0$, for Domain A (left) and Domain B (right). We set the polynomial degree $p=r=2$ in both space and time dimensions and time step $\Delta t = 0.01$.}
	\label{Fig:PlanewaveDisplacement}
\end{figure}


In Table~\ref{Tab:Comparison} we collect the condition number of the system matrix, the number of GMRES iterations and the execution time for the STDG$_0$ and STDG methods applied on a single time integration step, computed by using Domain A and Domain B, respectively. 
From the results we can observe that the proposed STDG method outperforms the STDG$_0$ one, in terms of condition number and GMRES iteration counts for the solution of the corresponding linear system. Clearly, for small problems, when the storage of the system matrix and the use of a direct solvers is possible the STSG$_0$ remains the most efficient solution.

	\begin{table}[h!]
	\centering
	\begin{tabular}{|l|c|c|c|c|c|c|c|}
		\hline
		\multirow{2}{*}{Dom.} & 
		\multirow{2}{*}{$p$} & 
		\multicolumn{2}{c|}{Condition number} & \multicolumn{2}{c|}{\# GMRES it.} & \multicolumn{2}{c|}{Execution time [s]} \\ \cline{3-8}
		&& STSG$_0$ & STDG & STSG$_0$ & STDG & STSG$_0$ & STDG \\
		\hline
		\hline
		A & 2 & $1.2\cdot10^9$ & $1.3\cdot10^2$ & $1.5\cdot10^4$ & $27$ & $1.1$ & $3.0\cdot10^{-3}$\\
		\hline
		A & 4 & $2.7\cdot10^{10}$ & $2.8\cdot10^3$ & $>10^6$ & $125$ & $>2200$ & $0.3\cdot10^{-1}$\\ 
		\hline
		B & 2 & $1.3\cdot10^{14}$ & $5.0\cdot10^2$ & $4.2\cdot10^5$ & $56$ & $452.3$ & $6.5\cdot10^{-2}$\\ 
		\hline
	\end{tabular}
	\caption{Test case of Section~\ref{Sec:PlaneWave}. Comparison between the proposed formulation \eqref{Eq:WeakProblem} and the method presented in \cite{AnMaMi20} in terms of conditioning and iterative resolution. We set $p=r$ and we fix the relative tolerance for the GMRES convergence at $10^{-12}$. }
	\label{Tab:Comparison}
\end{table}
 
\subsubsection{Layer over a half-space}
\label{Sec:LOH1}
	In this experiment, we test the performance of the STDG method by considering a benchmark test case, cf. \cite{DaBr01} for a real elastodynamics application, known in literature as layer over a half-space (LOH). We let $\Omega=(-15,15)\times(-15,15) \times(0,17)~{\rm km}$ be composed of two layers with different material properties, cf. Table~\ref{Table:LOH1Materials}. The domain is partitioned employing two conforming meshes of different granularity. The ``fine'' (resp. ``coarse'') grid is composed of $352800$ (resp. $122400$) hexahedral elements, varying from size $86~{\rm m}$ (resp. $167~{\rm m}$), in the top layer, to $250~{\rm m}$ (resp. $500~{\rm m}$) in the bottom half-space, cf. Figure~\ref{Fig:LOH1Domain}. On the top surface we impose a free surface condition, i.e. $\bsig \textbf{n} = \textbf{0}$, whereas on the lateral and bottom surfacews we consider absorbing boundary conditions \cite{stacey1988improved}.

\begin{figure} [h!]
	\centering
	\includegraphics[width=0.9\textwidth]{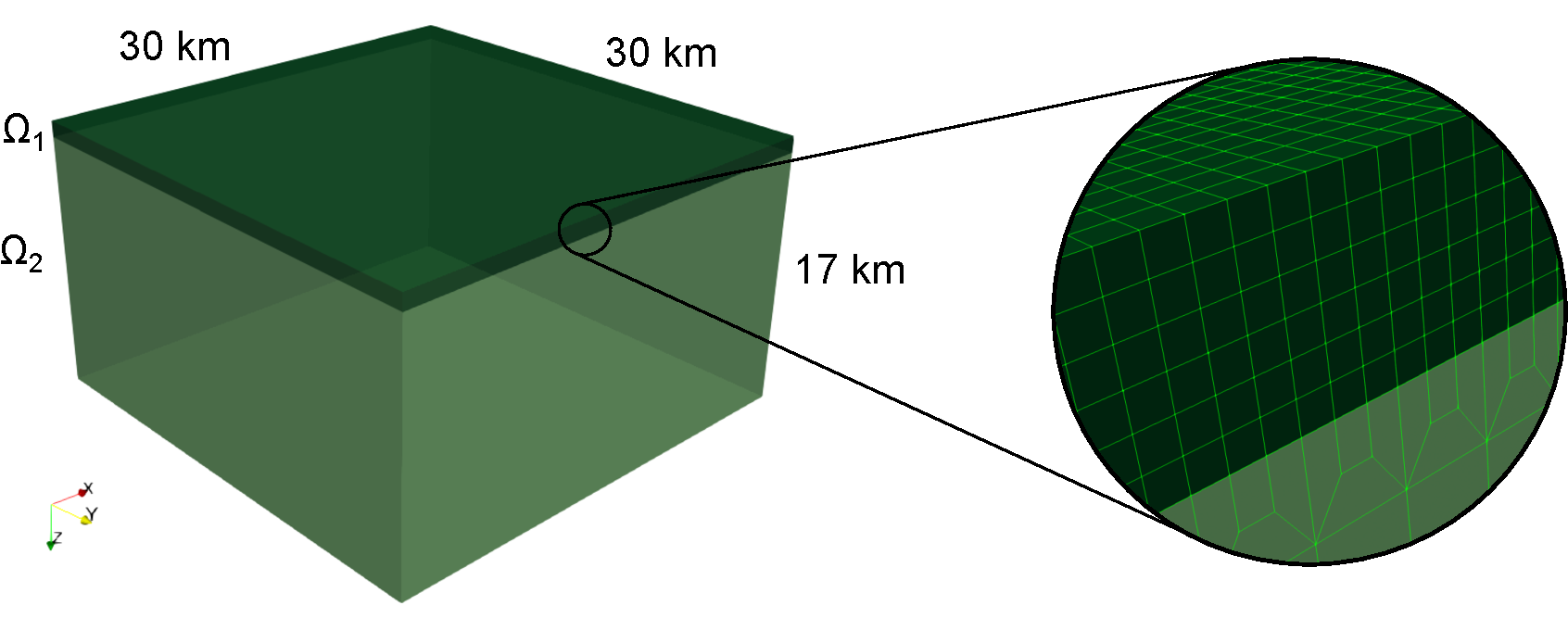}%
	\captionof{figure}{Test case of Section~\ref{Sec:LOH1}. Computational domain $\Omega = \cup_{\ell=1}^{2}\Omega_{\ell}$ and its partition.}
	\label{Fig:LOH1Domain}
\end{figure}

\begin{table}[h!]
	\centering
	\begin{tabular}{|l|r|r|r|r|r|}
		\hline
		Layer & Height $[km]$ & $\rho [kg/m^3]$ & $c_p [m/s]$ & $c_s [m/s]$ &  $\zeta [1/s]$ \\
		\hline
		\hline
		$\Omega_1$ & $ 1 $ & $2600$ & $2000$ & $4000$ & $0$ \\
		\hline
		$\Omega_2$ & $ 16 $ & $2700$ & $3464$ & $6000$ & $0$ \\
		\hline
	\end{tabular}
	\caption{Test case of Section~\ref{Sec:LOH1}. Mechanical properties of the medium. Here, the Lam\'e parameters $\lambda$ and $\mu$ can be obtained through the relations $\mu = \rho c_s^2$ and $\lambda = \rho c_p^2 -\mu$.}
	\label{Table:LOH1Materials}
\end{table}

The seismic excitation is given by a double couple point source located at the center of the domain expressed by 
\begin{equation}
\label{Eq:LOH1Source}
\ff(\xx,t) = \nabla \delta (\xx-\xx_S)M_0\bigg(\frac{t}{t_0^2}\bigg)\exp{(-t/t_0)},
\end{equation}
where $\xx_S = (0,0,2)~{\rm km}$, $M_0 = 10^8~{\rm Nm}$ is the scalar seismic moment, $t_0 = 0.1~{\rm s}$ is the smoothness parameter, regulating the frequency content and amplitude of the source time function. The semi-analytical solution is available in \cite{DaBr01} together with further details on the problem's setup.

We employ the STDG method with different choices of polynomial degrees and time integration steps. In Figures~\ref{Fig:LOH1ResultsFine41}-\ref{Fig:LOH1ResultsCoarse44-2} we show the velocity wave field computed at point $(6,8,0)~{\rm km}$ along with the reference solution, in both the time and frequency domains, for the sets of parameters tested. We also report relative seismogram error
\begin{equation}
\label{Eq:LOH1Error}
E = \frac{\sum_{i=1}^{n_S}(\uu_{\delta}(t_i)-\uu(t_i))^2}{\sum_{i=1}^{n_S}(\uu(t_i)^2)},
\end{equation}
where $n_S$ is the number of samples of the seismogram, $\uu_{\delta}(t_i)$ and $\uu(t_i)$ are, respectively, the value of seismogram at sample $t_i$ and the corresponding reference value. 	In Table~\ref{Table:LOH1Sensitivity} we report the set of discretization parameters employed, together with some results obtaineds in terms of accuracy and computational efficiency.

\begin{figure} [h!]
	\centering
	\includegraphics[width=0.5\textwidth]{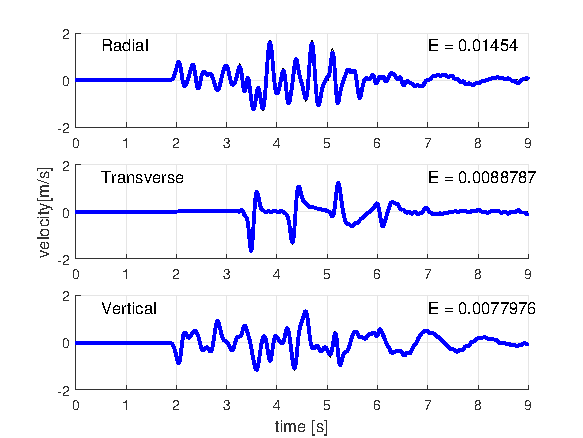}%
	\includegraphics[width=0.5\textwidth]{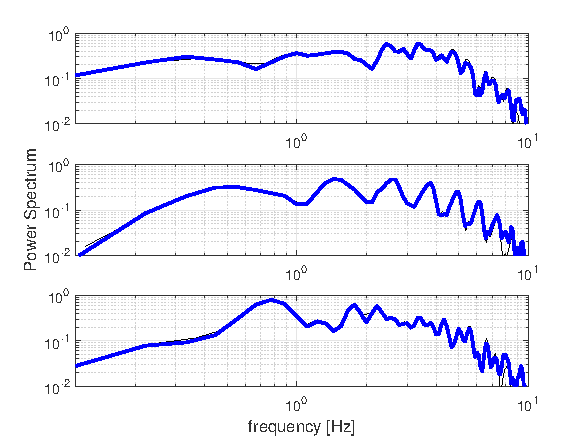}%
	\captionof{figure}{Test case of Section~\ref{Sec:LOH1}. Velocity wave field recorded at $(6,8,0)~{\rm km}$ along with the reference solution (black line), in the time domain (left) and frequency domain (right), obtained with the ``fine'' grid, polynomial degree $p=4$ for space and $r=1$ for time domain, and time-step $\Delta t = 10^{-3}~{\rm s}$. The error $E$ is computed as in \eqref{Eq:LOH1Error}.}
	\label{Fig:LOH1ResultsFine41}
\end{figure}

\begin{figure} [h!]
	\centering
	\includegraphics[width=0.49\textwidth]{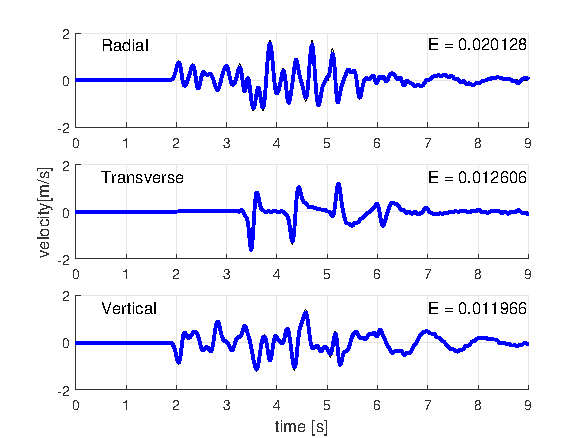}%
	\includegraphics[width=0.49\textwidth]{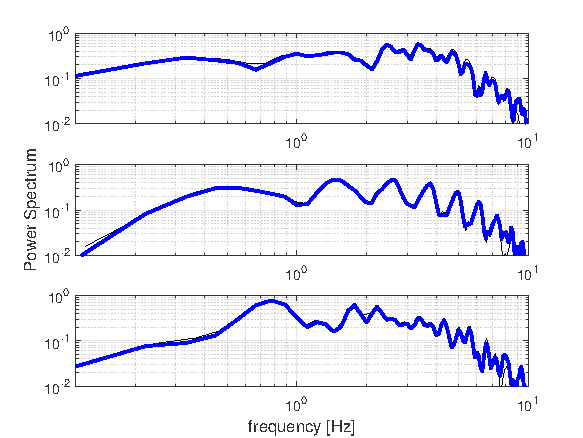}%
	\captionof{figure}{Test case of Section~\ref{Sec:LOH1}. Velocity wave field recorded at $(6,8,0)~{\rm km}$ along with the reference solution (black line), in the time domain (left) and frequency domain (right), obtained with the ``fine'' grid, polynomial degree $p=4$ for space and $r=2$ for time domain, and time-step $\Delta t = 10^{-3}~{\rm s}$. The error $E$ is computed as in \eqref{Eq:LOH1Error}.}
	\label{Fig:LOH1ResultsFine42}
\end{figure}

\begin{figure} [h!]
	\centering
	\includegraphics[width=0.49\textwidth]{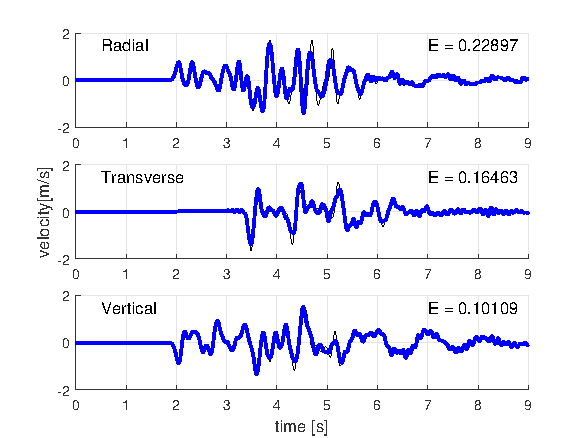}%
	\includegraphics[width=0.49\textwidth]{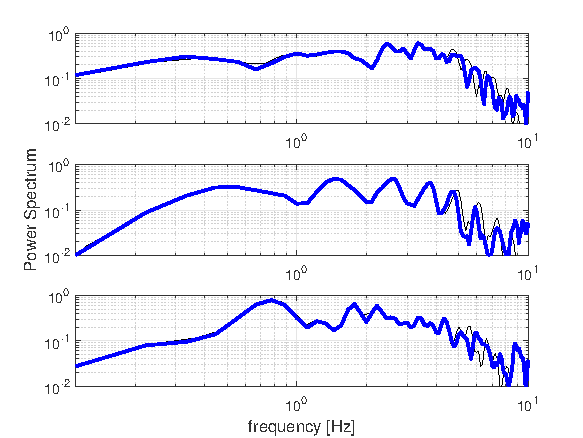}%
	\captionof{figure}{Test case of Section~\ref{Sec:LOH1}. Velocity wave field recorded at $(6,8,0)~{\rm km}$ along with the reference solution (black line), in the time domain (left) and frequency domain (right), obtained with the ``coarse'' grid, polynomial degree $p=4$ for space and $r=4$ for time domain, and time-step $\Delta t = 10^{-3}~{\rm s}$. The error $E$ is computed as in \eqref{Eq:LOH1Error}.}
	\label{Fig:LOH1ResultsCoarse44-3}
\end{figure}

\begin{figure} [h!]
	\centering
	\includegraphics[width=0.49\textwidth]{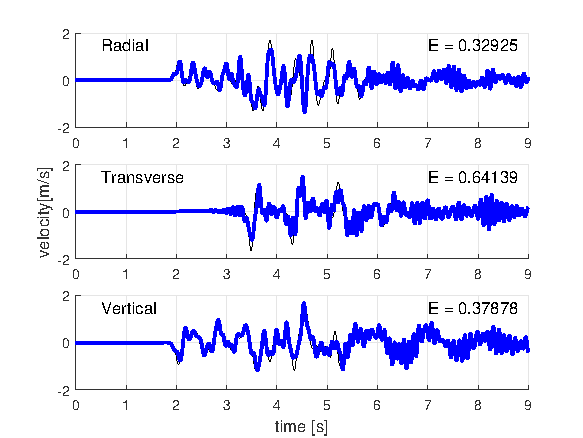}%
	\includegraphics[width=0.49\textwidth]{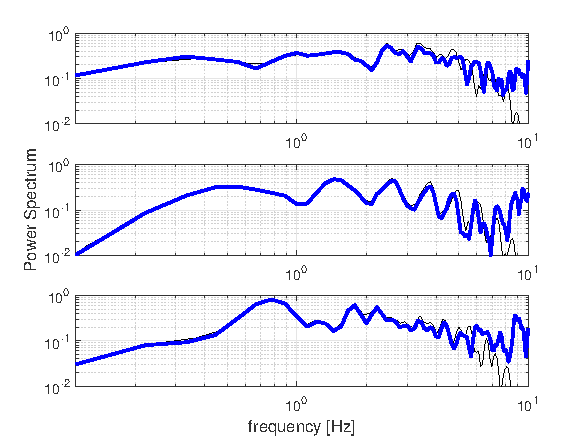}%
	\captionof{figure}{Test case of Section~\ref{Sec:LOH1}. Velocity wave field recorded at $(6,8,0)~{\rm km}$ along with the reference solution (black line), in the time domain (left) and frequency domain (right), obtained with the ``coarse'' grid, polynomial degree $p=4$ for space and $r=4$ for time domain, and time-step $\Delta t = 5\cdot10^{-2}~{\rm s}$. The error $E$ is computed as in \eqref{Eq:LOH1Error}.}
	\label{Fig:LOH1ResultsCoarse44-2}
\end{figure}	

\begin{table}[h!]
	\centering
	\begin{tabular}{|l|c|c|c|c|c|c|c|}
		\hline
		\multirow{2}{*}{Grid} & \multirow{2}{*}{$p$} & \multirow{2}{*}{$r$} & \multirow{2}{*}{$\Delta t~[{\rm s}]$} & GMRES &  Exec. Time & Tot. Exec. & \multirow{2}{*}{Error $E$}\\
		&&&&iter.&per iter. [s] &Time [s]&\\ 
		\hline
		\hline
		Fine & $4$ & $1$ & $10^{-3}$ & $6$ & $2.9$ & $3.08\cdot10^{4}$ & $0.015$ \\
		\hline
		Fine & $4$ & $2$ & $10^{-3}$ & $8$ & $5.6$ & $6.59\cdot10^{4}$ & $0.020$ \\
		\hline
		Coarse & $4$ & $4$ & $10^{-3}$ & $12$ & $7.6$ & $8.14\cdot10^{4}$ & $0.229$ \\
		\hline
		Coarse & $4$ & $4$ & $5\cdot10^{-2}$ & $24$ & $27.9$ & $7.22\cdot10^{4}$ & $0.329$ \\
		\hline
	\end{tabular}
	\caption{Test case of Section~\ref{Sec:LOH1}. Set of discretization parameters employed, and corresponding results in terms of computational efficiency and accuracy. The execution times are computed employing $512$ parallel processes, on \textit{Marconi100} cluster located at CINECA (Italy).}
	\label{Table:LOH1Sensitivity}
\end{table}

By employing the ``fine'' grid we obtain very good results both in terms of accuracy and efficiency. Indeed, the minimum relative error is less than $2\%$ with time polynomial degree $r=1$, see Figure~\ref{Fig:LOH1ResultsFine41}. Choosing $r=2$, as in Figure~\ref{Fig:LOH1ResultsFine42}, the error is larger (by a factor 40\%) but the solution is still enough accurate. However, in terms of total Execution time, with $r=1$ the algorithm performs better than choosing $r=2$, cf. Table~\ref{Table:LOH1Sensitivity}, column 7.
As shown in Figure~\ref{Fig:LOH1ResultsCoarse44-3}, the ``coarse'' grid produces larger errors and worsen also the computational efficiency, since the number of GMRES iterations for a single time step increases. Doubling the integration time step $\Delta t$, see Figure~\ref{Fig:LOH1ResultsCoarse44-2}, causes an increase of the execution time for a single time step that partly compensate the decrease of total number of time steps. Consequently, the total execution time reduces but only by 12\%. In addition, this choice causes some non-physical oscillations in the code part of the signal that contribute to increase the relative error. 
Indeed, we can conclude that for this test case, spatial discretization is the most crucial aspect. Refining the mesh produces a great decrease of the relative error and increases the overall efficiency of the method. Concerning time integration, it appears that the method performs well even with low order polynomial degrees both in terms of computational efficiency and  of accuracy. 
The method achieves its goal of accurately solving this elastodynamics problem that counts between 119 (``coarse'' grid) and 207 (``fine'' grid) millions of unknowns. The good properties of the proposed STDG method is once again highlighted by the fact that all the presented results are achieved without any preconditioning of the linear system. 

\subsection{Seismic wave propagation in the Grenoble valley}
\label{Sec:Grenoble}
In this last experiment, we apply the STDG method to a real geophysical study \cite{ChSt10}. This application consists in the simulation of seismic wave propagation generated by an hypothetical earthquake of magnitude $M_w = 6$ in the Grenoble valley, in the French Alps. The Y-shaped Grenoble valley, whose location is represented in Figure~\ref{Fig:GrenobleDomain}, is filled with late quaternary deposits, a much softer material than the one composing the surrounding mountains. We approximate the mechanical characteristics of the ground by employing three different material layers, whose properties are listed in Table~\ref{Table:GrenobleMaterials}. The alluvial basin layer contains  soft sediments that compose the Grenoble's valley and corresponds to the yellow portion of the domain in Figure~\ref{Fig:GrenobleDomain}. Then, the two bedrock layers approximate the stiff materials composing the surrounding Alps and the first crustal layer. The earthquake generation is simulated through a kinematic fault rapture along a plane whose location is represented in Figure~\ref{Fig:GrenobleDomain}.

\begin{figure} [h!]
	\centering
	\includegraphics[width=0.9\textwidth]{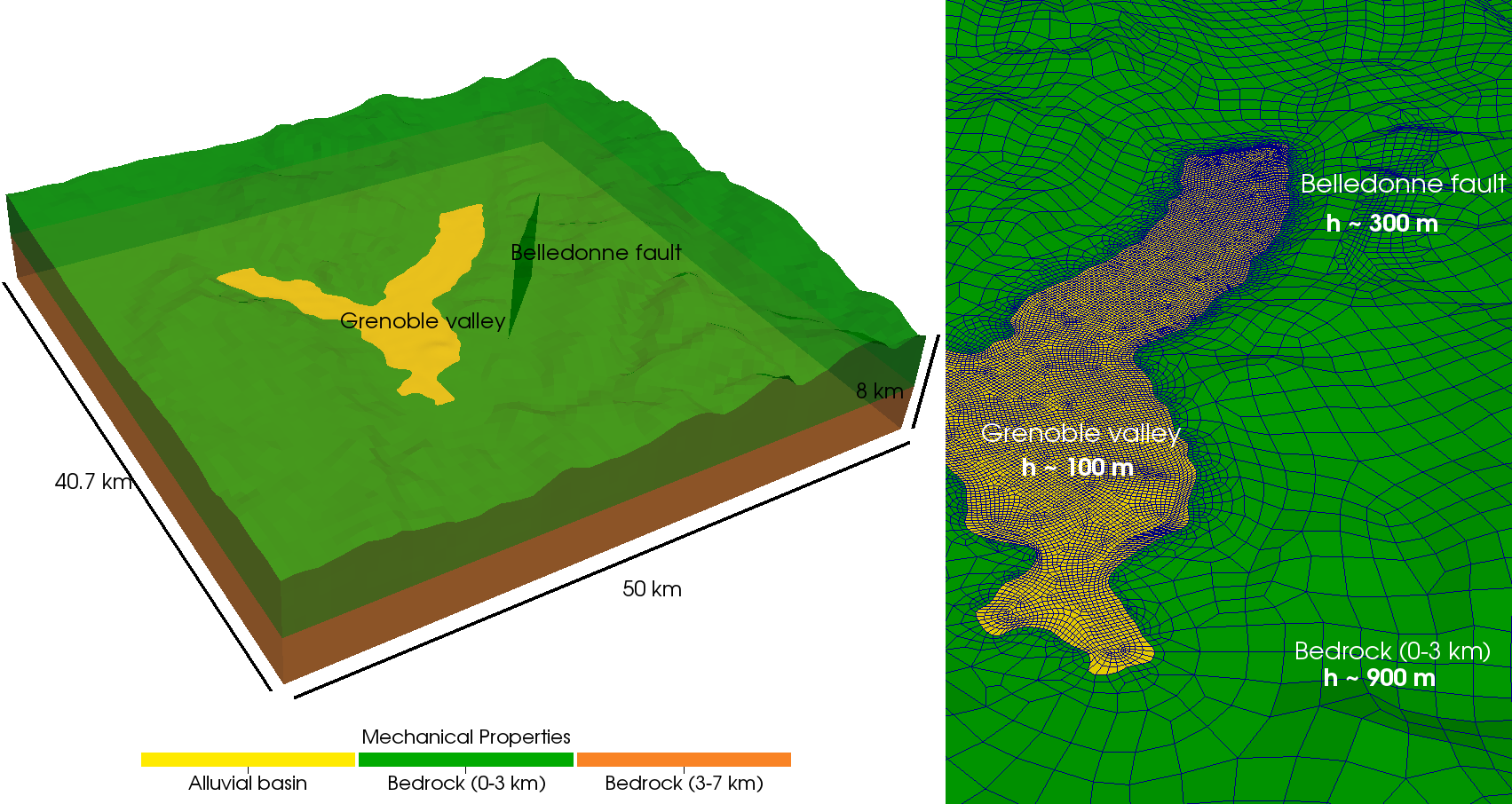}%
	\caption{Test case of Section~\ref{Sec:Grenoble}. Geophysical domain and its location.}
	\label{Fig:GrenobleDomain}
\end{figure}

\begin{table}[h!]
	\centering
	\begin{tabular}{|@{}|l|r|r|r|r|}
		\hline
		Layer &   $\rho~[{\rm kg/m^3}]$ & $c_s~[\rm{m/s}]$ & $c_p~[\rm{m/s}]$ & $ \zeta~[\rm {1/s}]$ \\
		\hline
		\hline
		Alluvial basin & 2140 + 0.125 $z_{d}$ & 300 + 19 $\sqrt{z_{d}}$  & 1450 + 1.2 $z_{d}$   & 0.01 \\
		\hline
		Bedrock $(0-3)$ km  & 2720 & 3200 & 5600 & 0 \\
		\hline
		Bedrock $(3-7)$ km & 2770 & 3430 & 5920 & 0 \\
		\hline
	\end{tabular}
	\caption{Test case of Section~\ref{Sec:Grenoble}. Mechanical properties of the medium. Here, the Lam\'e parameters $\lambda$ and $\mu$ can be obtained through the relations $\mu = \rho c_s^2$ and $\lambda = \rho c_p^2 -\mu$. $z_{d}$ measures the depth of a point calculated from the top surface.}
	\label{Table:GrenobleMaterials}
\end{table}

The computational domain $\Omega=(0,50)\times(0,47)\times (-7,3)~{\rm km}$ is discretized with a fully unstructured hexahedral mesh represented in Figure~\ref{Fig:GrenobleDomain}. The mesh, composed of $202983$ elements, is refined in the valley with a mesh size $h=100~{\rm m}$, while it is coarser in the bedrock layers reaching $h\approx 1~{\rm km}$.


\begin{figure} [h!]	
	\begin{minipage}{0.5\textwidth}
		\centering
		\includegraphics[width=\textwidth]{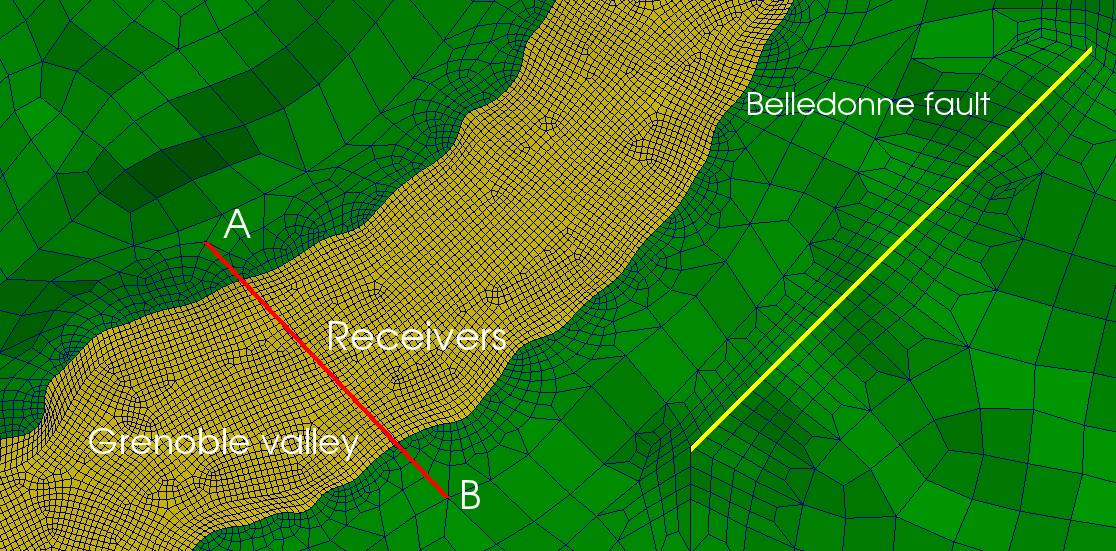}%
	\end{minipage}
	\begin{minipage}{0.5\textwidth}
		\centering
		\includegraphics[width=\textwidth]{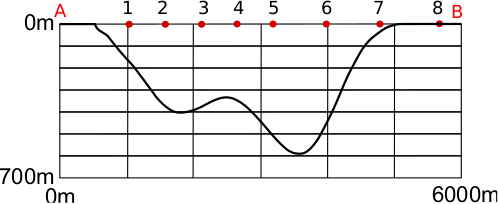}%
	\end{minipage}
	\caption{Left: surface topography in the Grenoble area. The white  line indicates the monitor points examined in Figure~\ref{Fig:GrenobleVel}. Right: cross section of the valley in correspondence of the monitor points.}
	\label{Fig:GrenoblePoints}
\end{figure}	

\begin{figure} [h!]
\begin{minipage}{\textwidth}
\centering
		\includegraphics[width=\textwidth]{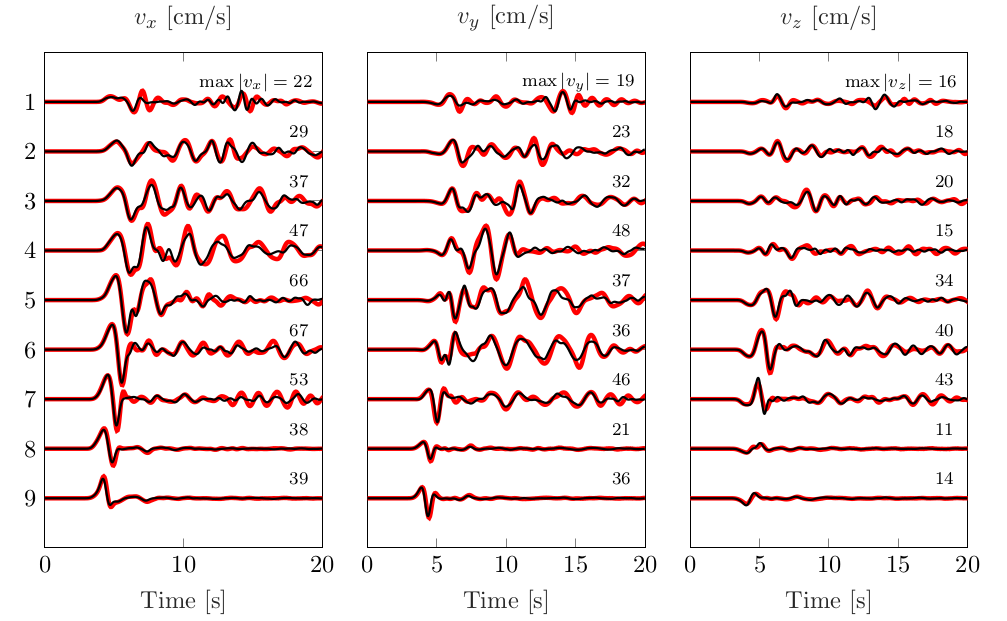}%
\end{minipage}
	\caption{Test case of Section~\ref{Sec:Grenoble}. Computed velocity field at the monitored points in Figure~\ref{Fig:GrenoblePoints}, together with the computed peak ground velocity for each monitor point.
	Comparisono between the STDG (bloack) solution and the SPECFEM (red) solution \cite{Chaljub2010QuantitativeCO}.}
	\label{Fig:GrenobleVel}
\end{figure}

On the top surface we impose a free surface condition, i.e. $\bsig \textbf{n} = \textbf{0}$, whereas on the lateral and bottom surface we consider absorbing boundary conditions \cite{stacey1988improved}. We employ the STDG method with polynomial degrees $p=3$ for the space discretization and $r=1$ for the time integration, together with a time step $\Delta t = 10^{-3}~{\rm s}$. We focus on a set of monitor points whose location is represented in Figure~\ref{Fig:GrenoblePoints}. In Figure~\ref{Fig:GrenobleVel}, we report the velocity field registered at these points compared with the ones obtained with a different code, namely SPECFEM   \cite{Chaljub2010QuantitativeCO}. The results are coherent with the different location of the points. Indeed, we observe highly perturbed waves in correspondence of the points $1-7$ that are located in the valley, i.e. in the alluvial material. This is caused by a refraction effect that arises when a wave moves into a soft material from a stiffer one. Moreover, the wave remains trapped inside the layer bouncing from the stiffer interfaces. The absence of this effect is evident from the monitors $8$ and $9$ that are located in the bedrock material. These typical behaviors are clearly visible also in Figure~\ref{Fig:GrenobleSnap}, where the magnitude of the ground velocity is represented for different time instants.
Finally, concerning the computation efficiency of the scheme, we report that, with this choice of discretization parameters, we get a linear system with approximately $36$ millions of degrees of freedom that is solved in $17.5$ hours, employing $512$ parallel processes, on \textit{Marconi100} cluster located at CINECA (Italy).

\begin{figure} [h!]
	\centering
	\includegraphics[width=0.49\textwidth]{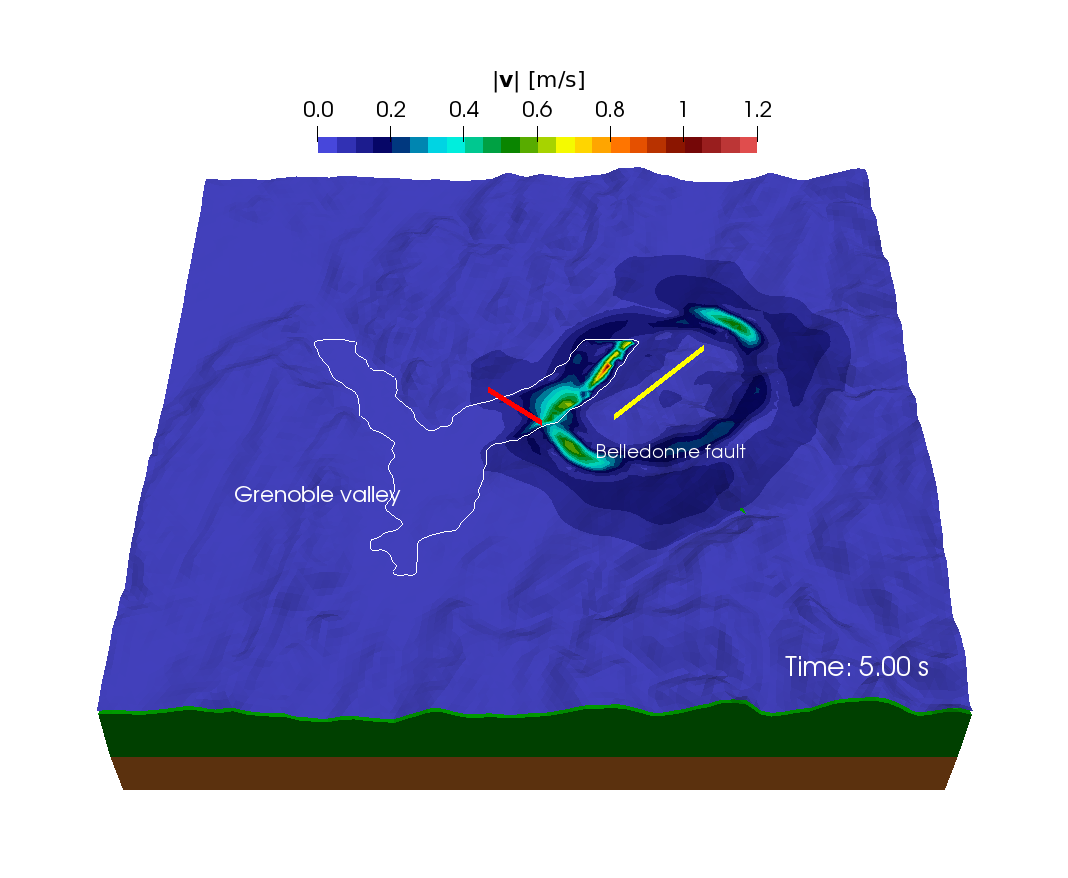}
		\includegraphics[width=0.49\textwidth]{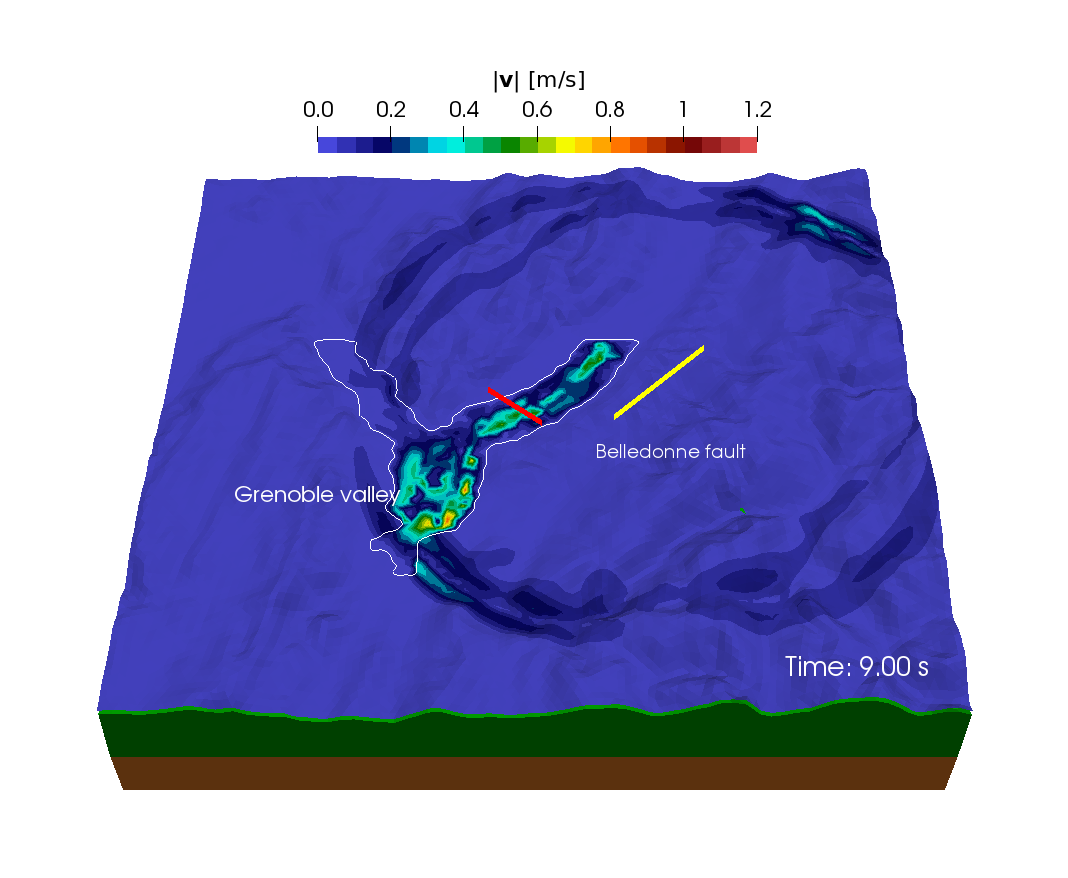}
			\includegraphics[width=0.49\textwidth]{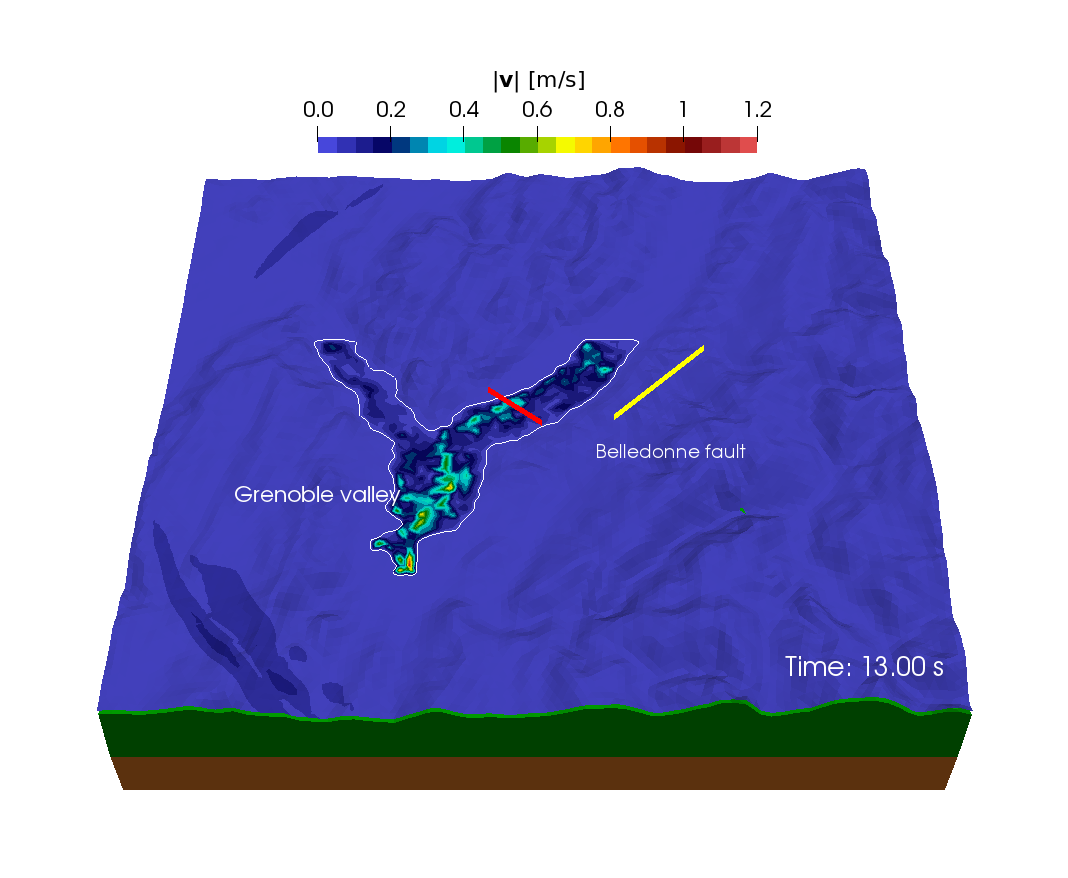}
				\includegraphics[width=0.49\textwidth]{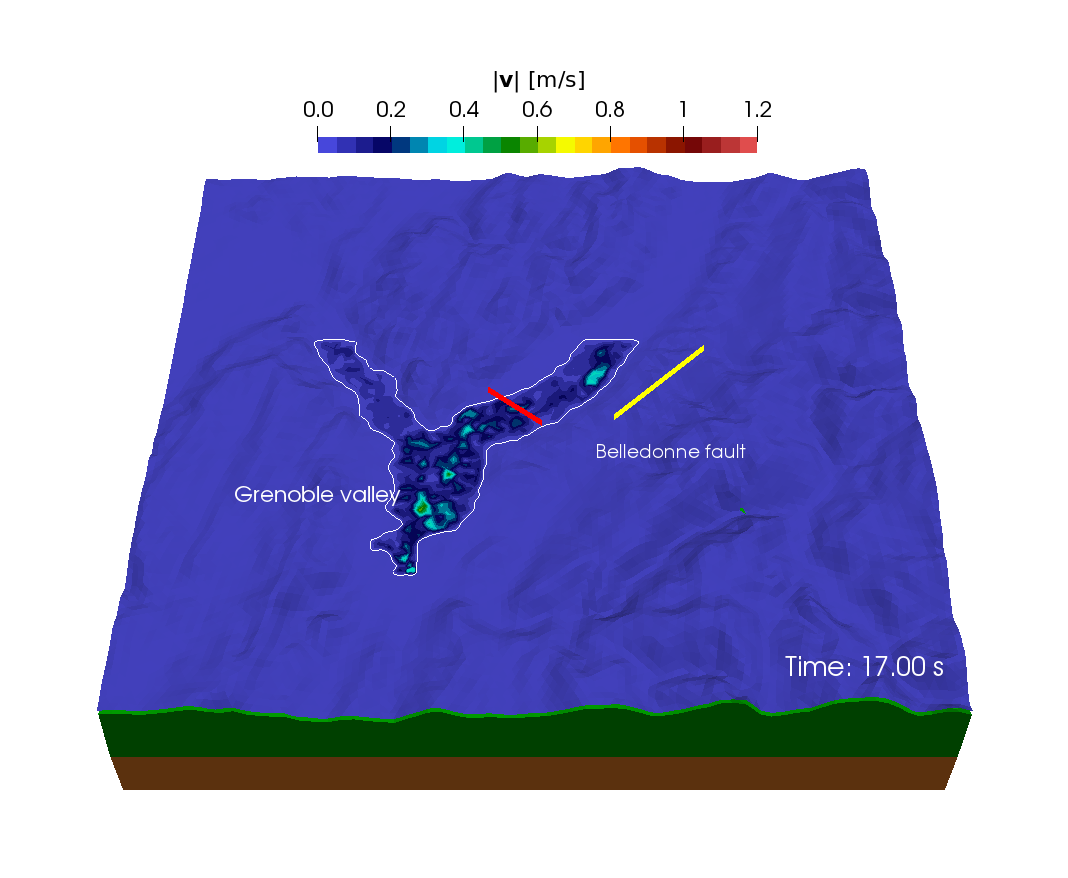}
	\caption{Test case of Section~\ref{Sec:Grenoble}. Computed ground velocity at different time instants obtained with polynomial degrees $p=3$ and $r=1$, for space and time, respectively, and $\Delta t = 10^{-3}~s$.}
	\label{Fig:GrenobleSnap}
\end{figure}

\section{Conclusions}
In this work we have presented and analyzed a new time Discontinuous Galerkin method for the solution of a system of second-order differential equations. We have built an energy norm that naturally arose by the variational formulation of the problem, and that we have employed to prove well-posedness, stability and error bounds. Through a manipulation of the resulting linear system, we have reduced the computation cost of the solution phase and we have implemented and tested our method in the open-source software SPEED (\url{http://speed.mox.polimi.it/}). Finally, we have verified and validated the proposed numerical algorithm through some two- and three-dimensional benchmarks, as well as real geophysical applications.

\section{Aknowledgements}
This work was partially supported by "National Group of Computing Science" (GNCS-INdAM). P.F. Antonietti has been supported by the PRIN research grant n. 201744KLJL funded by the Ministry of Education, Universities and Research (MIUR).

\end{document}